%% file: main.tex
\documentclass[oneside, a4paper]{amsart}
\usepackage[a4paper,  width=15cm, height=24cm]{geometry}
\usepackage{amsmath,amssymb, amsthm,amsfonts}
\usepackage{enumitem}
\usepackage[english]{babel}
\usepackage{tabularx}
\usepackage{tikz}
\usetikzlibrary{intersections, calc}
\usepackage{overpic}
\usepackage[auto]{contour}

\usepackage{graphicx}
\DeclareGraphicsExtensions{pdf, jpg, jpeg, png, pdf_tex}
\newtheoremstyle{dotless}{6pt}{6pt}{\em}{}{\bfseries}{.}{ }{}
\theoremstyle{dotless}
\newtheorem{definition}{Definition}[section]
\newtheorem{defi}[definition]{Definition}
\newtheorem{prop}[definition]{Proposition}

\newtheorem{lem}[definition]{Lemma}

\newtheorem{theorem}[definition]{Theorem}
\newtheorem{corollary}[definition]{Corollary}
\newtheorem{corollaryand}[definition]{Corollary and Definition}

\newtheorem*{remark*}{Remark}
\newcommand{\spann}[1]{\text{span} \{ #1 \}} 
\newcommand{\lspan}[1]{\langle{#1}\rangle}
\newcommand{\one}[1]{#1^{(1)}}
\newcommand{\two}[1]{#1^{(2)}}
\newcommand{\three}[1]{#1^{(3)}}

\newcommand{\n}{\mathfrak{n}}
\newenvironment{mycomment}
{\bigskip
  \hrule
  \medskip
\begin{raggedright}\color{blue}\it}
{\end{raggedright}
  \medskip
  \hrule
  \bigskip
}
\title{The Ribaucour families of discrete R-congruences}
\author{Thilo R\"orig and Gudrun Szewieczek}
\begin{document}

\contourlength{1pt}
\contournumber{32}

\maketitle
\begin{center}
\begin{minipage}{11cm}\small
\textbf{Abstract.} While a generic smooth Ribaucour sphere congruence admits exactly two envelopes, a discrete R-congruence gives rise to a 2-parameter family of discrete enveloping surfaces. The main purpose of this paper is to gain geometric insights into this ambiguity. In particular, discrete R-congruences that are enveloped by discrete channel surfaces and discrete Legendre maps with one family of spherical curvature lines are discussed.
\end{minipage}
\vspace*{0.5cm}\\\begin{minipage}{11cm}\small
\end{minipage}
\vspace*{0.5cm}\\\begin{minipage}{11cm}\small
\end{minipage}
\end{center}
%
%
\section{Introduction}

\noindent Classically, a smooth 2-parameter family of spheres is called \emph{Ribaucour sphere congruence} if  its two enveloping surfaces have corresponding curvature lines. Since a smooth sphere congruence admits at most two envelopes, each such sphere congruence provides a \emph{Ribaucour pair of surfaces}. Over the decades, this transformation concept was extended to submanifolds, partly also with singularities, \cite{MR2103311, MR2320656, MR1901374,saji2020behavior, MR2251001} and various (integrable) approaches for the construction were discussed \cite{MR2254053, MR2028680, MR2743444}. In these developments, special interest was paid to constrained Ribaucour transformations that allow to generate envelopes with special geometric properties, for example, the Darboux transformation of isothermic surfaces  and Ribaucour transformations that preserve  classes of $O$-surfaces \cite{O_surface}. 

A characteristic feature in the theory of smooth Ribaucour transformations are permutability theorems. In discrete differential geometry, those became crucial as underlying concept of integrable discretizations of curvature-line parametrized surfaces and orthogonal coordinate systems \cite{ddg_book, org_principles}. 

\medskip

\noindent Discrete equivalents of Ribaucour sphere congruences were developed in \cite{org_principles} and \cite{DOLIWA1999169}: \emph{discrete R-congruences} are provided by $Q$-nets in the Lie quadric, that is, a discrete congruence of spheres with planar faces. However, more flexibility in the discrete setup results in a 2-parameter family of discrete envelopes; we call it the \emph{Ribaucour family} of a discrete R-congruence (cf.\,Definition \ref{cor_initial_contact_el}).

In the last years, in a variety of works, discrete Ribaucour transformations between general discrete surfaces were discussed \cite{rib_coord, zbMATH01721625, zbMATH01327008}. Furthermore, restricted transformations for particular classes of discrete surfaces as discrete (special) isothermic surfaces \cite{ddg_book, discrete_cmc, zbMATH06406008, MR2004958, MR1676683} and discrete O-surfaces \cite{MR1997461} were investigated, to name just a few.

\medskip \noindent However, to the best of the authors' knowledge, there are no systematic studies of the geometry of the entire Ribaucour families in the literature, since usually pairs of discrete envelopes are considered as the main objects. 
Geometric investigations of this 2-parameter family of envelopes is the main contribution of this work. In most parts we consider the discrete R-congruences as the primary objects of interest and explore the corresponding Ribaucour families by using data provided by the sphere congruence.

\medskip

\noindent This paper is organized as follows. 
We recall basic principles and constructions in Lie sphere geometry in Section \ref{section_prel}.
Section \ref{section_r_congruences} is devoted to the construction of envelopes of discrete R-congruences.
As a crucial result, for any face of a discrete R-congruence, we introduce two involutive Lie inversions that map adjacent contact elements of discrete envelopes onto each other (Proposition~\ref{prop_inversions_for_r}). 
Using these inversions, we obtain a construction of a unique envelope from one prescribed initial contact element.
Hence, in this way, we can parametrize the entire Ribaucour family of a discrete R-congruence.

In Section \ref{section_special_congr} we restrict to discrete R-congruences with special properties and analyse their corresponding Ribaucour families.
This section also reveals relations to recent works on discrete Ribaucour coordinates \cite{rib_coord} and discrete $\Omega$-surfaces \cite{discrete_omega}.
Since we consider the discrete R-congruence as primary object, our approach sheds light on these situations from a different point of view.  

Focusing again on general discrete R-congruences, in Subsection \ref{subsec_rib_pairs}, we discuss how the choice of a facewise constant Lie inversion decomposes a Ribaucour family into pairs of envelopes.
Those Lie inversions also interact well with face-cyclides of these Ribaucour pairs and, therefore, induce a Ribaucour transformation of cyclidic nets as pointed out in Subsection \ref{subsec_cyclidic_nets}. 

As an application of the developed framework for discrete R-congruences, we investigate envelopes with one family of spherical curvature lines in Section \ref{section_spherical}. In particular, we characterize sphere congruences of Ribaucour families containing at least two discrete channel surfaces.

\medskip

\noindent\textbf{Acknowledgements.} The authors would like to thank Fran Burstall for helpful discussions about isothermic Q-nets. Moreover, financial support by JSPS Grant-in-Aid (as part of the FY2017 JSPS Postdoctoral fellowship) and TU Wien (Hörbiger Award) is gratefully acknowledged. 
This research was supported by the DFG Collaborative Research Center TRR
109 ``Discretization in Geometry and Dynamics''.

\section{Preliminaries}\label{section_prel}

\noindent In this section we briefly summarize some basic principles of Lie sphere geometry and recall some concepts that will be crucial for the rest of this work. For more details on this topic, the interested reader is referred to \cite{blaschke} and \cite{book_cecil}.

\medskip

\noindent Throughout the paper we use the hexaspherical model introduced by Lie and work in the vector space $\mathbb{R}^{4,2}$ endowed with the inner product $\lspan{\cdot\, , \cdot}$ given by 
\begin{equation*}
\lspan{ x , y } = -x_1y_1 + x_2y_2 +x_3y_3 +x_4y_4 +x_5y_5 -x_6y_6.
\end{equation*}
Homogeneous coordinates of elements in the projective space $\mathbb{P}(\mathbb{R}^{4,2})$ will be denoted by the corresponding black letter (i.e., oriented spheres $r,s \in \mathbb{P}(\mathcal{L})$ have homogeneous coordinates $\mathfrak{r}, \mathfrak{s} \in \mathbb{R}^{4,2}$ such that $\lspan{\mathfrak{r}} = r$ and $\lspan{\mathfrak{s}} = s$).
If statements hold for arbitrary homogeneous coordinates we do this without explicitly mentioning it.
The projective light cone is denoted by $\mathbb{P}(\mathcal{L})$ and represents the set of oriented 2-spheres in $\mathbb{S}^3$ in this model. Two spheres $r,s \in \mathbb{P}(\mathcal{L})$ are in oriented contact if and only if $\lspan{\mathfrak{r}, \mathfrak{s}}=0$.  Thus, the set $\mathcal{Z}$ of lines in $\mathbb{P}(\mathcal{L})$ corresponds to contact elements.

\medskip

\noindent By breaking symmetry, we can recover various subgeometries of Lie sphere geometry. Thus, let $\mathfrak{p} \in \mathbb{R}^{4,2}$ be a vector that is not lightlike, i.e., $\lspan{\mathfrak{p},\mathfrak{p}} \neq 0$. If $\mathfrak{p}$ is timelike, then $\langle\mathfrak{p}\rangle^\perp \cong
\mathbb{R}^{4,1}$ defines a Riemannian conformal geometry resp.\,M\"obius geometry. 
In this case, elements in $\mathbb{P}(\mathcal{L})\cap \langle \mathfrak{p} \rangle^\perp$ are considered points and will be called point spheres. 
If $\mathfrak{p}$ is spacelike, it determines a Lorentzian conformal geometry $\langle\mathfrak{p}\rangle^\perp \cong \mathbb{R}^{3,2}$ or planar Lie geometry.

%
%
\subsection{Linear systems and linear sphere complexes}
Let $s_1, s_2, s_3 \in \mathbb{P}(\mathcal{L})$ be three spheres that are mutually not in oriented contact, that is, $\lspan{\mathfrak{s_i}, \mathfrak{s_j}} \neq 0$ for $i\neq j \in \{1, 2, 3\}$. The space 
\begin{equation}
  \nonumber
  \spann{s_1, s_2, s_3} \cap \mathbb{P}(\mathcal{L})
\end{equation}
is called a \emph{linear system}\footnote{\cite[\S 53]{blaschke}: Lineare Kugelscharen und Kugelkomplexe} and the Lie invariant
\begin{equation}
  \nonumber
\delta:=\text{sign} \{ \langle \mathfrak{s}_1, \mathfrak{s}_2 \rangle \langle \mathfrak{s}_2, \mathfrak{s}_3 \rangle \langle \mathfrak{s}_3, \mathfrak{s}_1 \rangle  \}
\end{equation}
provides information about its geometry:
\begin{enumerate}[label=$\bullet$]
  \item If $\delta=-1$, then $s_1, s_2, s_3$ span a $(2,1)$-plane. Hence, the spheres of the linear system are curvature spheres of a Dupin cyclide and there exists a 1-parameter family of spheres in oriented contact with all spheres of the linear system -- the spheres of the orthogonal $(2,1)$-plane.

  \item If $\delta=1$, then $s_1, s_2, s_3$ span a $(1,2)$-plane. Thus, there does not exist a sphere that is in oriented contact with all spheres of the linear system, since the orthogonal complement is a $(3,0)$-plane that does not contain any spheres.
\end{enumerate}

\medskip

\noindent Any element $a \in \mathbb{P}(\mathbb{R}^{4,2})$ defines a \emph{linear sphere complex} $\mathbb{P}(\mathcal{L}) \cap a^\perp$, a 3-dimensional family of 2-spheres. Depending on the type of the vector $a$, we distinguish three cases:

\begin{enumerate}[label=$\bullet$]
\item \emph{parabolic} linear sphere complex, $\lspan{ \mathfrak{a}, \mathfrak{a}}=0$: all spheres in the complex are in oriented contact with the sphere defined by $a$, hence the linear sphere complex consists of all contact elements containing the sphere $a$.

\item \emph{hyperbolic} linear sphere complex, $\lspan{ \mathfrak{a}, \mathfrak{a}} < 0$: if we fix a M\"obius geometry by choosing $a$ as the point sphere complex, then the linear sphere complex consists of all point spheres. 

\item \emph{elliptic} linear sphere complex, $\lspan{ \mathfrak{a}, \mathfrak{a}} > 0$: in a M\"obius geometry, these linear sphere complexes then consist of all contact elements that intersect a fixed sphere at a fixed angle (see Fig.~\ref{fig:sphere_complex}, \emph{left}).
\end{enumerate}

\noindent Since elliptic linear sphere complexes will be essential in this paper, we will discuss their geometric construction illustrated in Fig.~\ref{fig:sphere_complex} in more detail. Assume that $a \in \mathbb{P}(\mathbb{R}^{4,2})$ determines an elliptic linear sphere complex. Additionally fix a parabolic complex $q$, as well as a hyperbolic complex $p$ to distinguish a M\"obius geometry modelled on $p^\perp$. 

Without loss of generality, we choose $\mathfrak{q}=(1,-1,0,0,0,0)$, $\mathfrak{p}=(0,0,0,0,0,1)$ and homogeneous coordinates $\mathfrak{a} = (\mathfrak{a}_i)_i \in \mathbb{R}^{4,2}$ of $a$ such that $\lspan{\mathfrak{a}, \mathfrak{q}}=1$.  By defining
\begin{equation*}
  c:=(\mathfrak{a}_3, \mathfrak{a}_4, \mathfrak{a}_5), \ 
  R:= \left\vert\sqrt{1 + \|c\|^2 - 2\mathfrak{a}_1}\right\vert 
  \text{, and } 
  r:= \lspan{\mathfrak{a}, \mathfrak{p}},
\end{equation*}
we obtain two concentric spheres $s_r$ and $s_R$ with center $c$ and radii $r$ and $R$, respectively (see Fig.~\ref{fig:sphere_complex}, \emph{left}). Then a straightforward computation shows that any sphere $s$ in the elliptic linear sphere complex $\lspan{\mathfrak{a}}^\perp$ intersects the sphere $s_R$ under the constant oriented angle $\cos \gamma= \frac{r}{R}$.
If $\lspan{\mathfrak{a}, \mathfrak{q}}=0$, the sphere $s_R$ corresponds to a plane and the sphere $s_r$ becomes the point at infinity. 

%
%
%
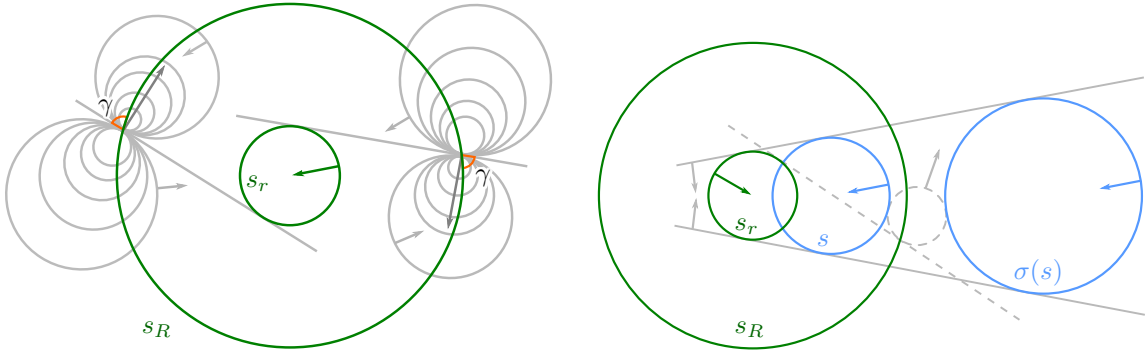
\begin{figure}[t]
  \begin{overpic}[width=.48\linewidth]{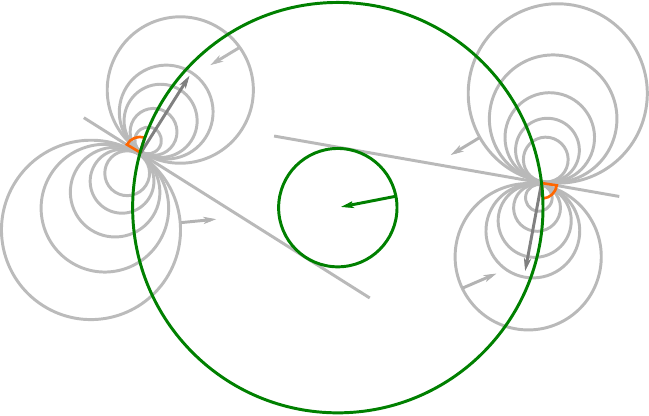}
    \put(44,30){\color{black!60!green}$s_r$}
    \put(25,3){\color{black!60!green}$s_R$}
    \put(17,44){\color{black}\contour{white}{$\gamma$}}
    \put(86,31){\color{black}\contour{white}{$\gamma$}}
  \end{overpic}
  \hfill
  \definecolor{myblue}{RGB}{85, 153, 255}
  \begin{overpic}[width=.48\linewidth]{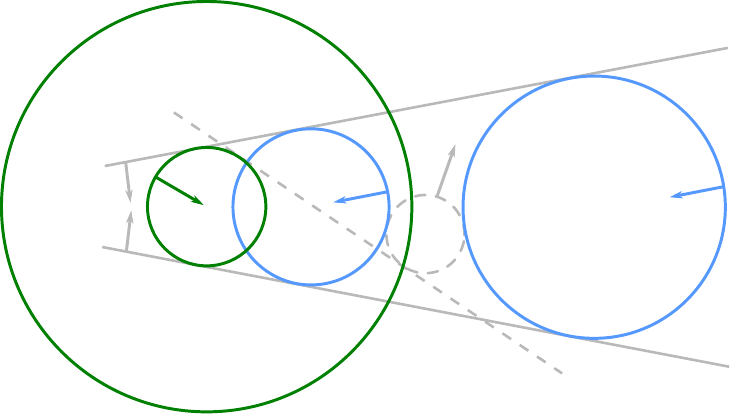}
    \put(25,22){\color{black!60!green}$s_r$}
    \put(40,19){\color{myblue}$s$}
    \put(25,3){\color{black!60!green}$s_R$}
    \put(76,13){\color{myblue}$\sigma(s)$}
  \end{overpic}
  \caption{
    \emph{Left:} The gray spheres lie in the elliptic sphere complex determined by the green spheres $s_R$ and $s_r$. All spheres of the complex intersect the sphere $S_R$ at a constant angle~$\gamma$. 
    \emph{Right:} The Lie inversion~$\sigma$ defined by the elliptic sphere complex determined by $s_R$ and $s_r$ maps the two blue spheres onto each other. The two gray lines and the dashed sphere are fixed by the involution. They determine the image~$\sigma(s)$ of the sphere~$s$.
  }
\label{fig:sphere_complex}
\end{figure}

\subsection{Lie inversions}
Let $a \in \mathbb{P}(\mathbb{R}^{4,2})$, $\lspan{ \mathfrak{a}, \mathfrak{a}} \neq 0$, then the \emph{Lie inversion with respect to the linear sphere complex $a$} is given by
\begin{equation}\label{equ_lie_inversion}
\sigma_\mathfrak{a}:\mathbb{R}^{4,2} \rightarrow \mathbb{R}^{4,2}, \ \ \sigma_\mathfrak{a}(\mathfrak{r}):=\mathfrak{r}-\frac{2\lspan{\mathfrak{r}, \mathfrak{a}}}{\lspan{\mathfrak{a}, \mathfrak{a}}}\mathfrak{a}.
\end{equation}
Lie inversions are involutive linear maps that preserve oriented contact between spheres and, therefore, also contact elements. Moreover, from the definition, we directly conclude that spheres contained in the linear sphere complex $a^\perp$ are fixed by the Lie inversion $\sigma_\mathfrak{a}$. 

For any pair of spheres that are not contained in the linear sphere complex $a^\perp$, we obtain the following property:

\begin{lem}\label{lem_inversion_lin_system}
Let $r, s \in \mathbb{P}(\mathcal{L})$ be two spheres that do not lie in the linear sphere complex $a^\perp$ with $\lspan{\mathfrak{a}, \mathfrak{a}} \neq 0$. Then, the $\spann{r, s, \sigma_{\mathfrak{a}}(r), \sigma_{\mathfrak{a}}(s)} \cap \mathbb{P}(\mathcal{L})$ is a linear system.
\end{lem}

\noindent The following two constructions give Lie inversions that map prescribed spheres onto each other: 

\begin{lem}\label{lem_inversion_family}
Let $\lspan{\mathfrak{r}} = r$ and  $\lspan{\bar{\mathfrak{r}}}= \bar{r}$ be fixed homogeneous coordinates of two spheres that are not in oriented contact and let $\sigma_\lambda$ denote the Lie inversion with respect to the linear sphere complex $\mathfrak{n}_\lambda:= \mathfrak{r} - \lambda\bar{\mathfrak{r}}$, where $\lambda \in \mathbb{R}^\times$. Then the following properties hold:
\begin{enumerate}[label=(\roman*)]
\item For any $\lambda \in \mathbb{R}^\times$, the Lie inversion $\sigma_\lambda$ maps the sphere $r$ to the sphere $\bar{r}$.
\item All spheres in oriented contact with $r$ \emph{and} $\bar{r}$ are fixed by any Lie inversion $\sigma_\lambda$.
\item A contact element that contains $r$ is mapped by any Lie inversion~$\sigma_\lambda$, $\lambda \in \mathbb{R}$, to the same contact element; however, different Lie inversions induce different correspondences between the spheres in the two contact elements. \label{lem_inversion_two}
\end{enumerate}
\end{lem}
\begin{proof}
The properties (i) and (ii) follow directly from equation (\ref{equ_lie_inversion}). To prove the third statement, we suppose that $v \in  \mathbb{P}(\mathcal{L})$ is a sphere in oriented contact with the sphere $r$ and $\lspan{\mathfrak{v}, \mathfrak{\bar{r}}}\neq 0$. Then, the sphere 
\begin{equation*}
\tilde{\mathfrak{s}}:={\lspan{\mathfrak{v}, \bar{\mathfrak{r}}}}\mathfrak{r} - {\lspan{\mathfrak{r}, \bar{\mathfrak{r}}}}\mathfrak{v} 
\end{equation*}
lies in the contact element $f:=\spann{r, v}$ and, furthermore, for any $\lambda \in \mathbb{R}$ in the linear sphere complex $\langle \mathfrak{n}_\lambda \rangle^\perp$. Hence, any Lie inversion $\sigma_\lambda$ preserves the sphere $\tilde{s}$ and maps the contact element~$f$ to the contact element $\sigma_\lambda(f)=\spann{\tilde{s}, \bar{r}}$.
\end{proof}

\begin{lem}\label{lem_unique_inversion_four_spheres}
Let $f$ and $\bar{f}$ be two contact elements sharing a common sphere $s \in \mathbb{P}(\mathcal{L})$. Then, for four spheres $r, t \in f$ and $\bar{r}, \bar{t} \in \bar{f}$  that do not coincide with $s$, there exists a unique Lie inversion $\sigma$ satisfying
\begin{equation*}
\sigma(\mathfrak{r})=\bar{\mathfrak{r}} \ \ \text{and} \ \ \sigma(\mathfrak{t})=\bar{\mathfrak{t}}.
\end{equation*}
\end{lem}
\begin{proof}
By assumption, there exist constants $\lambda, \mu, \bar{\lambda}, \bar{\mu} \in \mathbb{R}^\times$ such that
\begin{equation*}
\mathfrak{s}= \lambda \mathfrak{r} + \mu \mathfrak{t} \ \ \text{and} \ \ \mathfrak{s}= \bar{\lambda} \bar{\mathfrak{r}} + \bar{\mu} \bar{\mathfrak{t}}.
\end{equation*}
Then, the Lie inversion with respect to the linear sphere complex $a \in \mathbb{P}(\mathbb{R}^{4,2})$ determined by
\begin{equation*}
\mathfrak{a}:=\lambda \mathfrak{r} - \bar{\lambda} \bar{\mathfrak{r}} = \bar{\mu} \bar{\mathfrak{t}} - \mu \mathfrak{t}
\end{equation*}
provides the sought-after map.
Moreover, $a$ is the intersection of the lines $\spann{r,\bar{r}}$ and $\spann{t,\bar{t}}$ and hence unique.
\end{proof}

\noindent For later reference, we remark the following property for the composition of two Lie inversions that follows from straightforward computations: 

\begin{lem}\label{lem_inversions_commute}
Two Lie inversions commute, $\sigma_{\mathfrak{a}} \circ \sigma_{\mathfrak{b}}=\sigma_{\mathfrak{b}} \circ \sigma_{\mathfrak{a}}$, if and only if the corresponding linear sphere complexes are involutive, that is, $\lspan{\mathfrak{a},\mathfrak{b}}=0$. 
\end{lem}

\subsection{Cross-ratio of four spheres}
Similar to the cross-ratio of four concircular points, one defines the cross-ratio of four spheres in a linear system. We recall this definition and different ways to compute it.

\begin{defi}
Let $r_1, r_2, r_3, r_4 \in \mathbb{P}(\mathcal{L})$ be four spheres in a common contact element, then the \emph{cross-ratio} is defined by
\begin{equation*}
cr(\mathfrak{r}_1,\mathfrak{r}_2,\mathfrak{r}_3,\mathfrak{r}_4)=\frac{\alpha \bar{\beta}}{\beta \bar{\alpha}}, \ \text{ where} \ \mathfrak{r}_3=\alpha \mathfrak{r}_1 + \beta \mathfrak{r}_2 \text{ and } \mathfrak{r}_4= \bar{\alpha} \mathfrak{r}_1 + \bar{\beta} \mathfrak{r}_2.
\end{equation*}
\end{defi}
\noindent It is independent of the choice of homogeneous coordinates  and can be equivalently described by the four radii of the spheres
\begin{equation*}
cr(\mathfrak{r}_1,\mathfrak{r}_2,\mathfrak{r}_3,\mathfrak{r}_4)=\frac{(R_1-R_4)(R_2-R_3)}{(R_1-R_3)(R_2-R_4)},
\end{equation*}
where $R_i$ denotes the radius of the sphere $r_i$ (for point spheres the radius is assumed to be $0$ and for a plane we define the radius as $\infty$).

Moreover, one can define the \emph{cross-ratio of four spheres in a linear system} by tracking it back to the cross-ratio of four spheres in a contact element: assume that $r_1, r_2, r_3$ and $r_4$ are four spheres in a linear system and choose an arbitrary contact element $f \ni r_1$. Then there exist three unique spheres $s_i$ in $f$ such that $s_i \perp r_i$ for $i=2,3,4$ and we define
\begin{align*}
cr(r_1, r_2, r_3, r_4)&:= cr(r_1, s_2, s_3, s_4).
\end{align*}

%

\medskip

\noindent In particular, the cross-ratio of four spheres that are pairwise related by a Lie inversion is given by the following formula (cf.\,\cite[\S 53]{blaschke}):
\begin{lem}\label{lem_cr_J}
Let $a \in \mathbb{P}(\mathbb{R}^{4,2}) \setminus \mathbb{P}(\mathcal{L})$ and $s_1, s_2 \in \mathbb{P}(\mathcal{L})$ be two spheres that do not lie in the linear sphere complex $a^\perp$, then
\begin{equation*}
cr(\mathfrak{s}_1, \mathfrak{s}_2, \sigma_{\mathfrak{a}}(\mathfrak{s}_2) ,\sigma_{\mathfrak{a}}(\mathfrak{s}_1))=\frac{\phantom{2}\langle \mathfrak{s}_1, \mathfrak{s}_2 \rangle \langle \mathfrak{a}, \mathfrak{a} \rangle}{2\langle \mathfrak{s}_1, \mathfrak{a} \rangle \langle \mathfrak{s}_2, \mathfrak{a} \rangle}.
\end{equation*}
\end{lem}

\subsection{Discrete Legendre maps}
Discrete surfaces in this paper will be represented by discrete Legendre maps from a connected quadrilateral cell complex $\mathcal{G}=(\mathcal{V},\mathcal{E},\mathcal{F})$ of degree $4$ to the space of contact elements $\mathcal{Z}$. The set of vertices (0-cells), edges (1-cells) and faces (2-cells) of the cell complex $\mathcal{G}$ are denoted by $\mathcal{V}$, $\mathcal{E}$ and $\mathcal{F}$, respectively. 

The directions in the cell complex will be labelled by upper indices $(1)$ and $(2)$ and we obtain two distinguished sets of edges $\mathcal{E}^{(1)} \overset{.}{\cup} \mathcal{E}^{(2)}=\mathcal{E}$. A $(1)$- resp.\,$(2)$-coordinate ribbon is then the sequence of faces bounded by two adjacent $(1)$- resp.\,$(2)$-coordinate lines.

\medskip 

%
%
\begin{defi}[\cite{ddg_book, lin_weingarten_discrete}]
A discrete line congruence $f: \mathcal{V} \rightarrow \mathcal{Z}, i \mapsto f_i$, is a \emph{discrete Legendre map} if two adjacent contact elements $f_i$ and $f_j$ share a common \emph{curvature sphere} $s_{ij}:= f_i \cap f_j$.
\end{defi}
\noindent Note that, for any discrete Legendre map, we therefore obtain two \emph{curvature sphere congruences} $s^{(1)}:\mathcal{E}^{(1)} \rightarrow \mathbb{P}(\mathcal{L})$ and $s^{(2)}:\mathcal{E}^{(2)} \rightarrow \mathbb{P}(\mathcal{L})$.

Moreover, for any fixed point sphere complex $\mathfrak{p} \in \mathbb{R}^{4,2}$, $\lspan{\mathfrak{p}, \mathfrak{p}}=-1$, the point sphere congruence $p_i:= f_i \cap p^\perp$ of a discrete Legendre map provides a circular net, that is, any four point spheres of an elementary quadrilateral are concircular. 
%
%

\section{Discrete R-congruences}\label{section_r_congruences}

\noindent Classically, a smooth sphere congruence is called \emph{Ribaucour} if the curvature lines of its two envelopes correspond. Representing discrete surfaces by discrete Legendre maps, the analogous discrete problem, was studied in \cite{org_principles}: discrete sphere congruences enveloped by at least two generic discrete Legendre maps are given by \emph{$Q$-nets} in the Lie quadric, i.\,e.\,, sphere congruences $q:\mathcal{V} \rightarrow \mathbb{P}(\mathcal{L})$ such that the four spheres of any elementary quadrilateral are coplanar. 

In this paper, we will exclude degenerate faces of a $Q$-net and will assume that the spheres of a quadrilateral are not in oriented contact: 

\begin{defi}
A map $r:\mathcal{V} \rightarrow \mathbb{P}(\mathcal{L})$ is called a \emph{discrete R-congruence} if the four spheres of  any elementary quadrilateral lie in a unique linear system.
\end{defi}

\noindent As pointed out in \cite[\S 5]{org_principles}, discrete R-congruences with spheres lying in linear systems of signature $(2,1)$ provide the natural discrete counterparts to smooth Ribaucour sphere  congruences. Furthermore, since in this case any four spheres of a quadrilateral are curvature spheres of a Dupin cyclide, these $Q$-nets often allow for elegant and straightforward geometric interpretations. Thus, in this work,  we will often focus on these \emph{discrete $(2,1)$-R-congruences}.

\begin{remark*}
In the smooth case, Blaschke gave the following characterization of smooth Ribaucour sphere congruences \cite[\S 77]{blaschke}: a smooth sphere congruence $r:U \rightarrow \mathbb{P}(\mathcal{L})$ is a Ribaucour sphere congruence if and only if for any choice of coordinates $(u,v)$ there exists a map of \emph{osculating complexes}~$t:~U \rightarrow~\mathbb{P}(\mathbb{R}^{4,2})$ such that
\begin{equation*}
\spann{\mathfrak{r}, \ \partial_u \mathfrak{r}, \ \partial_v \mathfrak{r}, \ \partial_{uu} \mathfrak{r}, \ \partial_{vv} \mathfrak{r}, \ \partial_{uv} \mathfrak{r}} \subset  \mathfrak{t}^\perp.
\end{equation*} 
Similar osculating complexes also exist at vertices of a discrete R-congruence: by definition, the nine R-spheres of the four adjacent quadrilateral lie in a common linear sphere complex. This complex is spanned by the vertex sphere and its four neighbours.
\end{remark*}

\noindent Any face of a discrete R-congruence induces two special Lie inversions that will turn out to be crucial in the construction of its envelopes (see Fig.~\ref{fig:schematic_quad} \textit{right}):

\begin{prop}\label{prop_inversions_for_r}
A discrete R-congruence induces a unique map of linear sphere complexes determined by  
\begin{equation*}
(\one{n}{},\two{n}{}):\mathcal{F} \rightarrow \mathbb{P}(\mathbb{R}^{4,2})\setminus \mathbb{P}(\mathcal{L}) \times \mathbb{P}(\mathbb{R}^{4,2})\setminus \mathbb{P}(\mathcal{L})
\end{equation*}
such that for any quadrilateral $(ijkl)$ the corresponding Lie inversions $\sigma_{\mathfrak{n}^{(1)}}$ and $\sigma_{\mathfrak{n}^{(2)}}$ satisfy
\begin{equation}
\begin{aligned}\label{eq:prop_n}
&\sigma_{\mathfrak{n}^{(1)}}(\mathfrak{r}_i)=\mathfrak{r}_j, \ \ \sigma_{\mathfrak{n}^{(1)}}(\mathfrak{r}_l)=\mathfrak{r}_k \ \ \text{ and }
\\&\sigma_{\mathfrak{n}^{(2)}}(\mathfrak{r}_i)=\mathfrak{r}_l, \ \ \sigma_{\mathfrak{n}^{(2)}}(\mathfrak{r}_j)=\mathfrak{r}_k;
\end{aligned}
\end{equation}
these induced Lie inversions will be denoted by 
\begin{equation*}
\one{\sigma}:=\sigma_{\mathfrak{n}^{(1)}} \ \ \text{and } \ \ \two{\sigma}:=\sigma_{\mathfrak{n}^{(2)}}.
\end{equation*}
\end{prop}
\begin{proof}
  Let $r_i, r_j, r_k$ and $r_l$ be four spheres of an elementary quadrilateral of a discrete R-congruence (see Fig.~\ref{fig:schematic_quad}). Since the four spheres lie in a linear system and are therefore linearly dependent, we can choose homogeneous coordinates such that
\begin{equation}\label{choice_coord_spheres}
\mathfrak{r}_i - \mathfrak{r}_j + \mathfrak{r}_k - \mathfrak{r}_l =0.
\end{equation} 
Then, by Lemma \ref{lem_inversion_family} (i), the vectors 
\begin{equation}\label{equ_coord_compl}
\one{\mathfrak{n}}:= \mathfrak{r}_i-\mathfrak{r}_j=\mathfrak{r}_l-\mathfrak{r}_k \ \text{ and } \ \two{\mathfrak{n}}:=\mathfrak{r}_i - \mathfrak{r}_l = \mathfrak{r}_j-\mathfrak{r}_k
\end{equation}
define two linear sphere complexes satisfying~Equation \eqref{eq:prop_n}.
\end{proof}

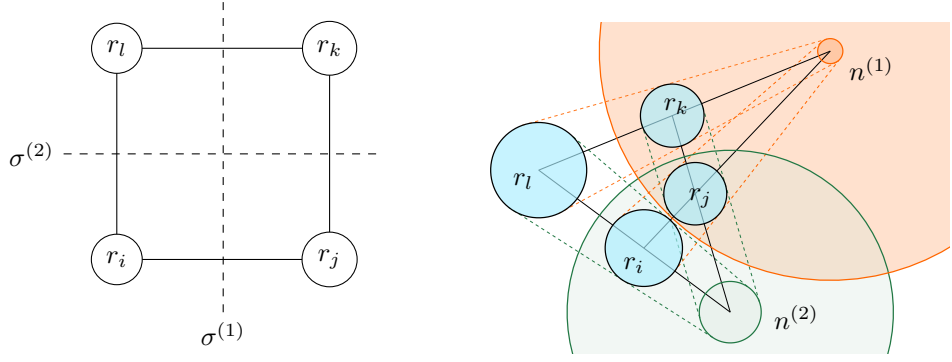
\begin{figure}[t]
  \input{A-left.tikz}
  \hspace{0.5cm}
  \begin{overpic}[width=6.6cm]{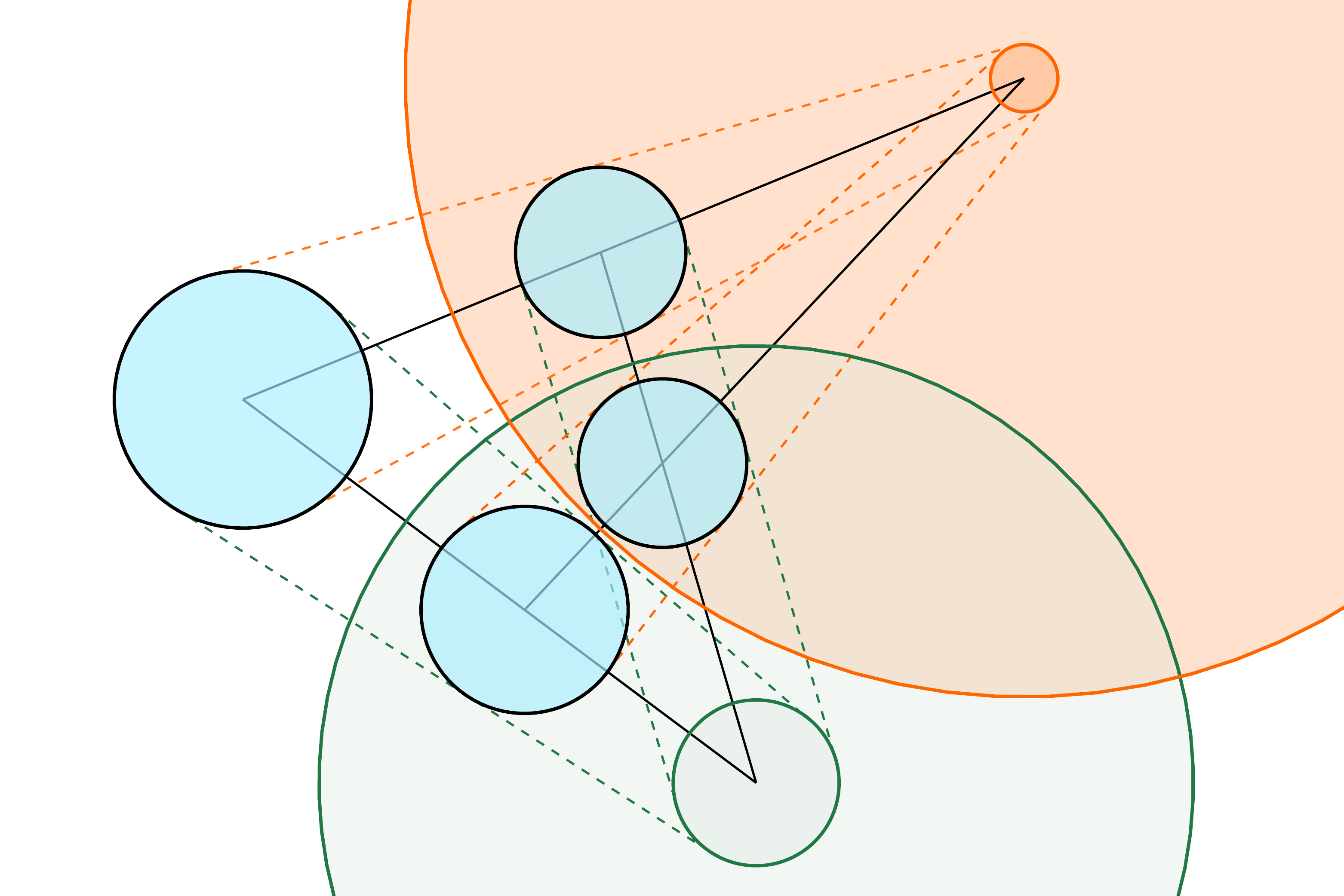}
    \put(13,34){$r_l$}
    \put(35,17){$r_i$}
    \put(48,31){$r_j$}
    \put(43,49){$r_k$}
    \put(80,55){$n^{(1)}$}
    \put(65,5){$n^{(2)}$}
  \end{overpic}
  \caption{\emph{Left}: A schematic picture of the two Lie inversions $\sigma^{(1)}$ and $\sigma^{(2)}$ of an R-congruence quadrilateral $(r_i, r_j, r_k, r_l)$.  \emph{Right}: A plane section through an R-congruence quadrilateral showing the two elliptic complexes $n^{(1)}$ and $n^{(2)}$ of the inversions $\sigma^{(1)}$ and $\sigma^{(2)}$.
  }
  \label{fig:schematic_quad}
\end{figure}

The choice of homogeneous coordinates in Equation \eqref{choice_coord_spheres} is local and depends on the quadrilateral under consideration. However, for special discrete R-congruences there exist global homogeneous coordinates inducing the Lie inversions $\one{\sigma}$ and $\two{\sigma}$ (see Subsection \ref{section_iso}).

\begin{lem}\label{lem_involutive}
For any quadrilateral, the two linear sphere complexes determined by $n^{(1)}$ and $n^{(2)}$ are involutive.
\end{lem}
\begin{proof}
Without loss of generality, for any quadrilateral we choose homogeneous coordinates as given in Equation~\eqref{equ_coord_compl}. Then
\begin{align*}
2 \langle \mathfrak{n}^{(2)} , \mathfrak{n}^{(1)}\rangle &= \langle \mathfrak{r}_i-\mathfrak{r}_l, \mathfrak{r}_i - \mathfrak{r}_j \rangle + \langle \mathfrak{r}_j-\mathfrak{r}_k , \mathfrak{r}_l - \mathfrak{r}_k \rangle
\\&= \langle \mathfrak{r}_j , -\mathfrak{r}_i - \mathfrak{r}_k + \mathfrak{r}_l \rangle + \langle \mathfrak{r}_l, - \mathfrak{r}_i+ \mathfrak{r}_j - \mathfrak{r}_k \rangle =0.
\end{align*}
\end{proof}
 
\begin{lem}
Let $(r_i, r_j, r_k, r_l)$ be four spheres of an elementary quadrilateral of a discrete R-congruence. 
\begin{enumerate}[label=(\roman*)]
\item $cr(r_i, r_j, r_k, r_l)< 0$ if and only if $\one{n}$ and $\two{n}$ determine two elliptic or two hyperbolic linear sphere complexes.
\item $cr(r_i, r_j, r_k, r_l)> 0$ if and only if the linear sphere complexes determined by $\one{n}$ and $\two{n}$ are of different type, that is, one linear sphere complex is elliptic and the other one is hyperbolic.
\end{enumerate}
\end{lem}

\begin{proof}
Suppose that $(r_i, r_j, r_k, r_l)$ are four spheres of an elementary quadrilateral of a discrete R-congruence and choose homogeneous coordinates such that $0=\mathfrak{r}_i-\mathfrak{r}_j+\mathfrak{r}_k-\mathfrak{r}_l$. Then, by Lemma~\ref{lem_cr_J}, we obtain that 
\begin{equation}\label{equ_cr}
 cr(r_i, r_j, r_k, r_l)= 
 \frac{
   \phantom{2}
   \langle \mathfrak{r}_i, \mathfrak{r}_j \rangle 
   \langle \mathfrak{\two{n}}, \mathfrak{\two{n}} \rangle
 }
 {
 2\langle \mathfrak{r}_i, \mathfrak{\two{n}} \rangle 
 \langle \mathfrak{r}_j, \mathfrak{\two{n}} \rangle
 }
 =
 -\frac{
   \lspan{\one{\mathfrak{n}},\one{\mathfrak{n}}}
 }{
   \lspan{\two{\mathfrak{n}},\two{\mathfrak{n}}}
 }
\end{equation}
and therefore conclude that
\begin{align*}
cr(r_i, r_j, r_k, r_l) < 0 \ &\Leftrightarrow \ \text{sign} \{{\lspan{\mathfrak{\one{n}}, \mathfrak{\one{n}}}}\} = \text{sign} \{{\lspan{\mathfrak{\two{n}}, \mathfrak{\two{n}}}}\}.
\end{align*}
The points $\one{n}$ and $\two{n}$ lie in the plane spanned by the spheres ${r}_i, {r}_j, {r}_k$ and ${r}_l$. Therefore, if the plane has signature $(2,1)$, the linear sphere complexes are elliptic; if the spheres lie in a $(1,2)$-plane, then the linear sphere complexes are hyperbolic.

Moreover, $cr(r_i, r_j, r_k, r_l) > 0$ if and only if 
\begin{equation*}
\text{sign} \{{\lspan{\mathfrak{\two{n}}, \mathfrak{\two{n}}}}\}=-\text{sign} \{{\lspan{\mathfrak{\one{n}}, \mathfrak{\one{n}}}}\},
\end{equation*}
that is, if and only if the two linear sphere complexes are of different type (see Fig.~\ref{fig:embedded_quad} for the M\"obius geometric picture).
\end{proof}

This proposition also includes the special case of four concircular point spheres. It is well-known that a circular quadrilateral is embedded if and only if the cross-ratio is negative. In this case, the Lie inversions $\one{\sigma}$ and $\two{\sigma}$ become M\"obius inversions (see Figure 3).

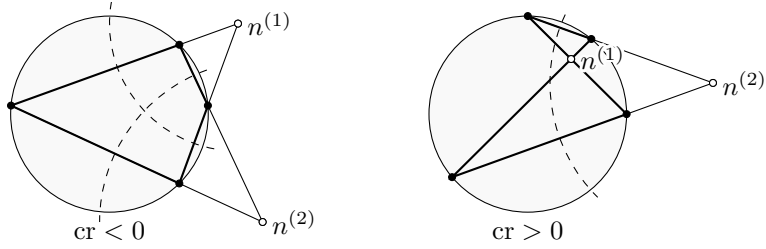
\begin{figure}[tb]
  \input{C-embedded.tikz}
  \hspace{1cm}
  \input{C-non_embedded.tikz}
  \caption{
    Whether a circular quadrilateral is embedded or non-embedded depends on the sign of the cross-ratio of the circular quadrilateral. The sign also determines the type of the M\"obius inversion (resp.\ linear sphere complex in Lie geometry) associated to the quadrilateral.  
  }
  \label{fig:embedded_quad}
\end{figure}

\medskip

\noindent With the Lie inversions at hand, we are now prepared to discuss the main objects of interest, namely the envelopes of a discrete R-congruence:

\begin{defi}
  A discrete Legendre map $f:\mathcal{V}\rightarrow \mathcal{Z}$ is the \emph{envelope} of a discrete R-congruence $r:\mathcal{V}\rightarrow \mathbb{P}(\mathcal{L})$, if $r_i \in f_i$ for all $i \in \mathcal{V}$. 
\end{defi}

\noindent Envelopes of a discrete R-congruence can be constructed from one prescribed initial contact element using the inversions defined in Proposition \ref{prop_inversions_for_r}. Thus, suppose that $f_i:=\spann{s_0, r_i}$ is an arbitrary initial contact element at the vertex $i \in \mathcal{V}$. Then, the contact elements 
\begin{equation*}
\mathfrak{f}_j:=\one{\sigma}(\mathfrak{f}_i) , \ \ \mathfrak{f}_k:=\two{\sigma}(\mathfrak{f}_j), \ \ \mathfrak{f}_l:=\one{\sigma}(\mathfrak{f}_k)=\two{\sigma}(\mathfrak{f}_i)
\end{equation*} 
define an envelope for the face $(ijkl)$ of the discrete R-congruence $r$: firstly, observe that by Lemmas~\ref{lem_inversions_commute} and \ref{lem_involutive}, the contact elements $\one{\sigma}(\mathfrak{f}_k)$ and $\two{\sigma}(\mathfrak{f}_i)$ indeed coincide. Furthermore, from  Proposition~\ref{prop_inversions_for_r} and Lemma~\ref{lem_inversion_family}, we deduce that the contact elements envelop the discrete R-congruence,
\begin{equation*}
r_j \in f_j, \ \ r_k \in f_k \ \ \text{and } \ r_l \in f_l,
\end{equation*}
and two adjacent ones intersect.

Moreover, due to Lemma~\ref{lem_inversion_family}\ref{lem_inversion_two}, this construction uniquely extends to all vertices $\mathcal{V}$ of the quadrilateral cell complex. 

\medskip

\noindent Conversely, any envelope of a discrete R-congruence arises in this way:
\begin{lem}\label{transport_contact}
Let $f:\mathcal{V} \rightarrow \mathcal{Z}$ be a discrete Legendre map enveloping the discrete R-congruence $r:\mathcal{V}\rightarrow \mathbb{P}(\mathcal{L})$, then for any quadrilateral $(ijkl)$ we obtain that
\begin{align*}
&\one{\sigma}(\mathfrak{f}_i)=\mathfrak{f}_j, \ \ \one{\sigma}(\mathfrak{f}_l)=\mathfrak{f}_k \ \ \text{ and } \ \ \two{\sigma}(\mathfrak{f}_i)=\mathfrak{f}_l, \ \ \two{\sigma}(\mathfrak{f}_j)=\mathfrak{f}_k.
\end{align*}
\end{lem}

\begin{proof}
  Assume that $f$ is an envelope, then we know that $r_i \in f_i$ and $r_j \in f_j$ (see Fig.~\ref{fig:contact_elements} \emph{Left} for the labelling of a quadrilateral). Therefore, the Lie inversion $\one{\sigma}$ preserves the curvature sphere $s_{ij}=f_i \cap f_j$.

Hence, for any sphere $\mathfrak{r}_i + \lambda\mathfrak{s}_{ij} \in \mathfrak{f}_i$ in the contact element, it follows that
\begin{align*}
\one{\sigma}(\mathfrak{r}_i + \lambda\mathfrak{s}_{ij})= \one{\sigma}(\mathfrak{r}_i) + \lambda \one{\sigma}(\mathfrak{s}_{ij})= \mathfrak{r}_j + \lambda \mathfrak{s}_{ij} \in \mathfrak{f}_j,
\end{align*}
which completes the proof.
\end{proof}

\noindent Thus, we have reproven the following existence result for envelopes of a discrete R-congruence that was already given in \cite[Theorem~3.37]{ddg_book}: 

\begin{corollaryand}\label{cor_initial_contact_el}
Any envelope of a discrete R-congruence is uniquely determined by the choice of an initial contact element and, therefore, any discrete R-congruence admits a 2-parameter family of envelopes. 

This family of envelopes is said to be the \emph{Ribaucour family of a discrete R-congruence}. If the discrete R-congruence consists of faces with signature $(2,1)$, the family is also called a \emph{$(2,1)$-Ribaucour family}.
\end{corollaryand}

\noindent 
For later reference we remark on some relations between the envelopes and the Lie inversions associated to the discrete R-congruence (see Fig.~\ref{fig:contact_elements}, \emph{Right}):

\begin{lem}\label{lem_curv_spheres_swapped}
Let $f:\mathcal{V} \rightarrow \mathcal{Z}$ be a discrete Legendre map enveloping the discrete R-congruence $r:\mathcal{V}\rightarrow \mathbb{P}(\mathcal{L})$. Then, for any quadrilateral, the Lie inversions $\one{\sigma}$ and $\two{\sigma}$ swap the curvature spheres on opposite edges: 
\begin{equation*}
\two{\sigma}(\mathfrak{s}_{ij})=\mathfrak{s}_{kl} \ \ \text{and} \ \ \one{\sigma}(\mathfrak{s}_{jk})=\mathfrak{s}_{li}.
\end{equation*}
\end{lem}

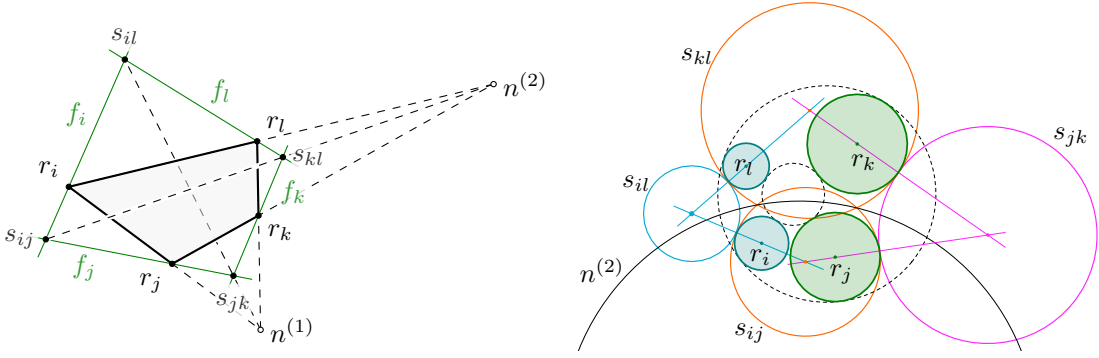
\begin{figure}[tb]
  \newcommand{\tikzscale}{1}
  \input{B-linecongruence_quad.tikz}
  \begin{overpic}[width=7cm]{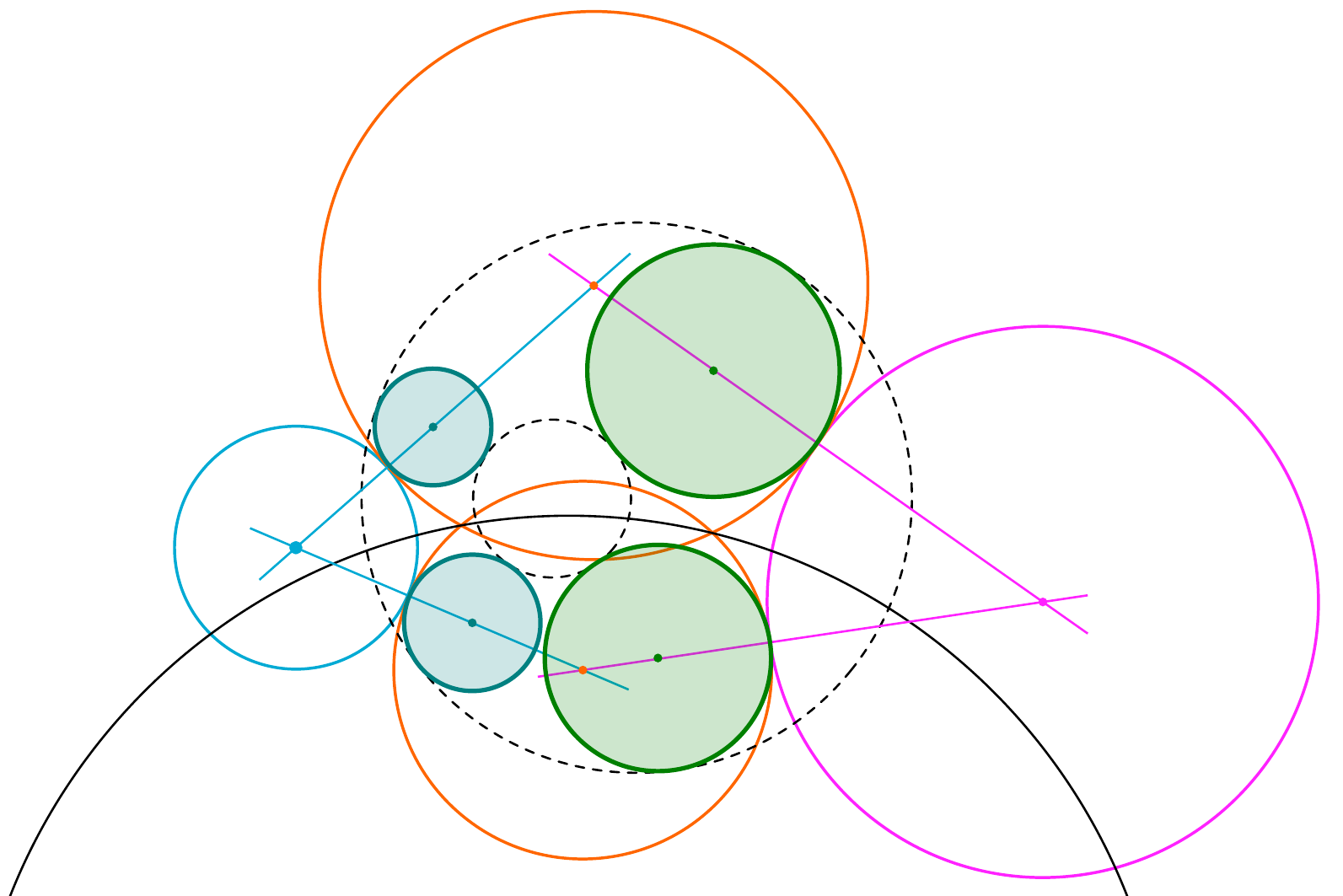}
    \put(33,17){$r_i$}
    \put(48,14){$r_j$}
    \put(52,35){$r_k$}
    \put(30,34){$r_l$}
    \put(1,13){$n^{(2)}$}
    \put(9,31){$s_{il}$}
    \put(90,40){$s_{jk}$}
    \put(30,3){$s_{ij}$}
    \put(20,55){$s_{kl}$}
  \end{overpic}
  \caption{
    \emph{Left:} An R-congruence quadrilateral $(r_i, r_j, r_k, r_l)$ with contact elements~$f$, curvature spheres~$s$, and the points~$\one{n}, \two{n}$ defining the complexes.
    \emph{Right:} The elements with the same color are mapped onto each other by the involution $\sigma^{(2)}$ induced by the complex $n^{(2)}$, e.g., $r_i \leftrightarrow r_l$, $s_{kl} \leftrightarrow s_{ij}$.
}
\label{fig:contact_elements}
\end{figure}

\subsection{M\"obius geometric point of view}\label{subsection_moebius}
In this subsection we briefly discuss the interplay between the construction of Ribaucour families in Lie sphere and M\"obius geometry. Thus, we fix a point sphere complex $\mathfrak{p} \in \mathbb{R}^{4,2}$, $\lspan{\mathfrak{p},\mathfrak{p}}=-1$, to distinguish a M\"obius geometry modelled on $\lspan{\mathfrak{p}}^\perp$.

Since a Ribaucour family consists of discrete Legendre maps, the choice of a point sphere complex reveals a 2-parameter family of enveloping circular principal contact element nets. The underlying circular nets are called the \emph{point sphere envelopes in $\lspan{\mathfrak{p}}^\perp$} (see Fig.~\ref{fig:circular_envelopes}).
 
\begin{figure}[b]
  \includegraphics[width=8cm]{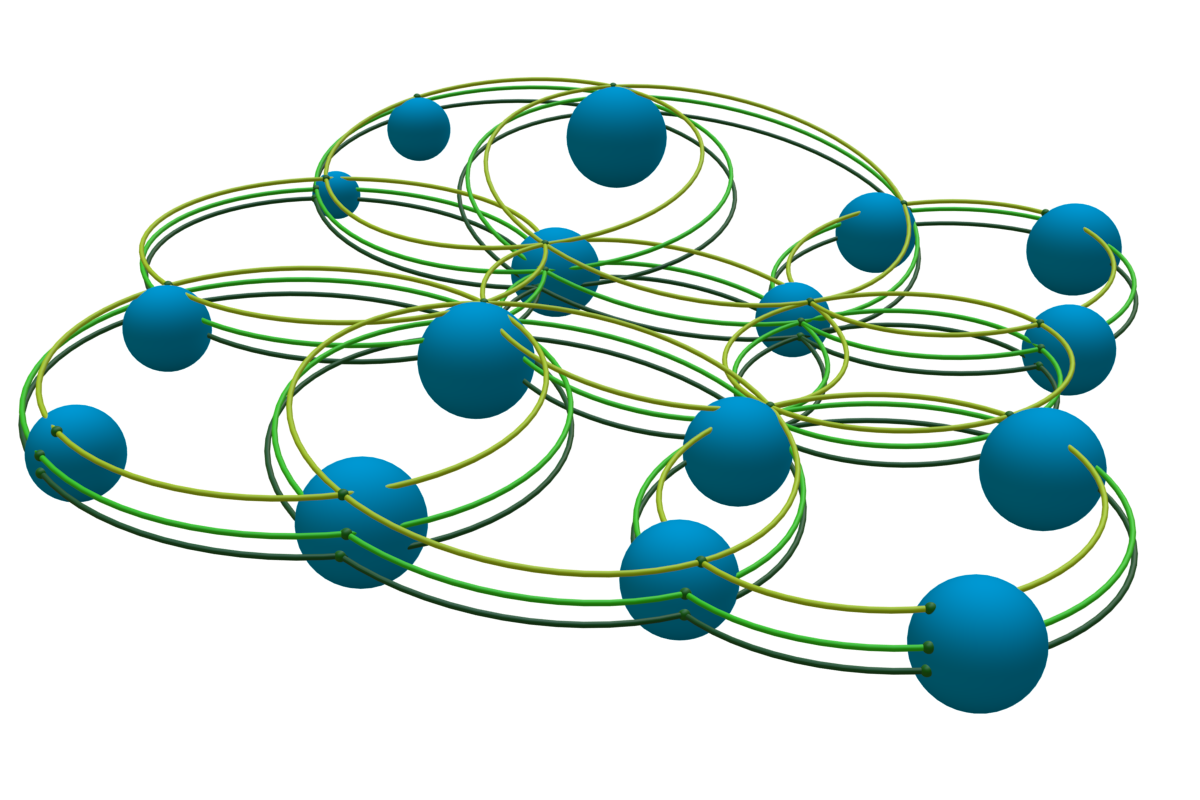}
  \caption{
    A discrete R-congruence with various
    envelopes (plotted as circular quadrilaterals) of the Ribaucour family.
  }
  \label{fig:circular_envelopes}
\end{figure}

\medskip

\noindent Note that a Lie inversion $\sigma_{\mathfrak{a}}$ is a M\"obius transformation, that is, it maps point spheres onto point spheres, if and only if $\lspan{\mathfrak{a},\mathfrak{p}}=0$. Therefore, generically, the Lie inversions $\one{\sigma}$ and $\two{\sigma}$ that map adjacent enveloping contact elements onto each other, do not transport point spheres along the discrete R-congruence. 

However, for each edge there exists a M\"obius transformation uniquely determined by data from the discrete R-congruence that relates point spheres of adjacent contact elements in the Ribaucour family:

\begin{prop}\label{prop_moebius_inversion_edge}
Let $r: \mathcal{V} \rightarrow \mathbb{P}(\mathcal{L})$ be a discrete R-congruence. Then the Lie inversions $\sigma_\mathfrak{m}$ with respect to the linear sphere complexes determined by
\begin{equation*}
m:\mathcal{E} \rightarrow \mathbb{R}^{4,2}, \  m_{ij}=\lspan{\mathfrak{p}, r_j}r_i - \lspan{\mathfrak{p}, r_i}r_j
\end{equation*} 
are M\"obius transformations and satisfy, for any edge $(ij)$,
\begin{equation*}
\sigma_\mathfrak{m_{ij}}(\mathfrak{r}_i)=\mathfrak{r}_j.
\end{equation*} 
Moreover, those M\"obius transformations transport the contact elements of envelopes along the discrete R-congruence and map adjacent point spheres of the point sphere envelope onto each other.
\end{prop}

\subsection{Construction of discrete R-congruences}
\label{subsection_construction}
Discrete R-congruences can be constructed from two arbitrary initial curves of spheres defined on two intersecting coordinate lines $\one{I}=(\one{\mathcal{V}}, \one{\mathcal{E}})$ and $\two{I}=(\two{\mathcal{V}}, \two{\mathcal{E}})$ of a quadrilateral cell complex $\mathcal{G}$ by the following iteration: 
\begin{itemize}
\item fix two initial sphere curves $\one{c}: \one{\mathcal{V}} \rightarrow \mathbb{P}(\mathcal{L})$ and $\two{c}: \two{\mathcal{V}} \rightarrow \mathbb{P}(\mathcal{L})$ that intersect at one vertex
\item choose for each face $(ijkl)$ that contains an edge $(ij)$ of the curve $\one{c}$ a Lie inversion that maps the prescribed spheres $c_i$ and $c_j$ onto each other; for each face this amounts to the choice of a $\lambda \in \mathbb{R}$:
\begin{equation*}
\one{n}_{(ijkl)}:=\one{\mathfrak{c}}_i - \lambda\one{\mathfrak{c}}_j
\end{equation*}
\item starting at a face $(ijkl)$ with three prescribed spheres $\one{c}_i$, $\one{c}_j$ and $\two{c}_l$, we complete the face by defining
\begin{equation*}
\mathfrak{r}_k:=\one{\sigma}_{(ijkl)}(\two{\mathfrak{c}}_l)
\end{equation*}
\item iteratively this procedure gives a coordinate ribbon of a discrete R-con\-gruence including the spheres of the curve $\one{c}$ and two spheres of the curve $\two{c}$
\item to obtain the next coordinate ribbon we choose again suitable Lie inversions along the just constructed sphere curve and proceed as described above.
\end{itemize}

\begin{remark*}
  The choice of the parameter $\lambda$ in the construction above is equivalent to choosing the cross-ratios for the quadrilaterals. The relation is given by Lemma~\ref{lem_cr_J}:
  \begin{align*}
    cr(\mathfrak{c}_i, \mathfrak{c}_l, \one{\sigma}_{(ijkl)}(\mathfrak{c}_l), \one{\sigma}_{(ijkl)}(\mathfrak{c}_i))
    = \frac
  {\lspan{\mathfrak{c}_i, \mathfrak{c}_l}}
    {\lspan{\mathfrak{c}_i, \mathfrak{c}_l} - \lambda\lspan{\mathfrak{c}_j, \mathfrak{c}_l}}\,.
  \end{align*}
\end{remark*}
\medskip

Conversely, given a nowhere umbilic discrete Legendre map $f: \mathcal{V} \rightarrow \mathbb{P}(\mathbb{R}^{4,2})$, that is, opposite curvature spheres do not coincide, then any choice of spheres in the contact elements along two intersecting coordinate lines uniquely determines a discrete R-congruence. The construction relies on the following simple observation:

\begin{lem}
  \label{lem:r_congruence_from_legendre}
Given four contact elements $(f_i, f_j, f_k, f_l)$ of an elementary quadrilateral of a nowhere umbilic discrete Legendre map and three spheres $r_i \in f_i$, $r_j \in f_j$ and $r_l \in f_l$, then there exists a unique sphere $r_k \in f_k$ such that $(r_i, r_j, r_k, r_l)$ provides a quadrilateral of a discrete R-congruence.
\end{lem}
\begin{proof}
By Lemma \ref{lem_unique_inversion_four_spheres}, there exists a unique Lie inversion $\sigma$ satisfying
\begin{equation*}
\sigma(\mathfrak{r}_i)=\mathfrak{r}_l \ \text{ and } \ \sigma(\mathfrak{s}_{ij})=\mathfrak{s}_{kl}.
\end{equation*}
Then, $\mathfrak{r}_k:=\sigma(\mathfrak{r}_j) \in \spann{\mathfrak{s}_{jk}, \mathfrak{s}_{kl}}$ provides the sought-after sphere that completes the elementary quadrilateral of a discrete R-congruence. 
\end{proof}

As a consequence of this Lemma, we obtain the following construction for discrete R-congruences of a discrete Legendre map: we fix two intersecting coordinate lines $\one{I}=(\one{\mathcal{V}}, \one{\mathcal{E}})$ and $\two{I}=(\two{\mathcal{V}}, \two{\mathcal{E}})$ of the quadrilateral cell complex $\mathcal{G}$ and suppose that $\one{c}: \one{\mathcal{V}} \rightarrow \mathbb{P}(\mathcal{L})$ and $\two{c}: \two{\mathcal{V}} \rightarrow \mathbb{P}(\mathcal{L})$ are chosen such that
\begin{equation*}
\one{c}_i \in f_i \text{ for any } i \in \one{\mathcal{V}} \ \text{and } \ \one{c}_j \in f_j \text{ for any } j \in \two{\mathcal{V}}.
\end{equation*}
Then, by Lemma~\ref{lem:r_congruence_from_legendre} above, these choices uniquely determine a discrete R-congruence of the discrete Legendre map $f$. 
This is another instance of the interplay between line congruences and Q-nets as studied in~\cite{DoliwaSantiniManas:2000:TransformationsOfQnets, bobenko_schief-discrete_line_complexes}.



\section{Special discrete R-congruences}\label{section_special_congr}

\noindent This section is devoted to discrete R-congruences with special properties. We discuss their geometry and consequences for the corresponding Ribaucour families.

\subsection{Discrete R-congruences in a fixed linear sphere complex}\label{subsection_fixed_complex}

Suppose that $r: \mathcal{V} \rightarrow \mathbb{P}(\mathcal{L})$ is a discrete R-congruence lying in a fixed linear sphere complex determined by $a \in \mathbb{P}(\mathbb{R}^{4,2})$, that is, $\lspan{\mathfrak{r}_i, \mathfrak{a}}=0$ for all $i \in \mathcal{V}$. Since the spheres all lie in the complex, so do $\one{n}$ and $\two{n}$ and the linear sphere complex $a^\perp$ is invariant under $\one{\sigma}$ and $\two{\sigma}$.
Depending on the type of the linear sphere complex we obtain three geometrically different situations.

\smallskip

\subsubsection*{Hyperbolic complex}
\noindent In the case of a {hyperbolic} linear sphere complex we can fix $a \in \mathbb{P}(\mathbb{R}^{4,2})$, $\lspan{\mathfrak{a}, \mathfrak{a}} < 0$, to be the point sphere complex and consider the M\"obius geometry modelled on $\lspan{\mathfrak{a}}^\perp$. Then, on each quadrilateral the R-spheres become point spheres lying in a linear system with signature $(2,1)$. Therefore, the R-spheres are curvature spheres of a Dupin cyclide consisting of point spheres and the discrete R-congruence becomes a circular net. In this case, the Lie inversions $\one{\sigma}$ and $\two{\sigma}$ become M\"obius inversions mapping the point spheres of the circular net onto each other.

For each circular net there exists a 2-parameter choice of contact elements that yield principal contact element nets. This ambiguity in the choice of the contact elements is reflected in the Ribaucour family of $r$. 

\subsubsection*{Parabolic complex} 
If $a \in \mathbb{P}(\mathcal{L})$ determines a {parabolic} linear sphere complex, then all R-spheres of the congruence are in oriented contact with the sphere determined by $a$.  
According to Corollary \ref{cor_initial_contact_el}, the choice of an initial contact element determines a unique envelope. 
Let us choose a particular initial contact element at the sphere $r_i$ given by $f_i:=\spann{r_i, a}$. Since the sphere determined by $a$ is preserved by the Lie inversions  $\one{\sigma}$ and $\two{\sigma}$, this envelope is totally umbilic, that is, all curvature spheres of the discrete Legendre map coincide. 

Moreover, observe that, for any projection to a M\"obius geometry, the point sphere map of a totally umbilic envelope lies on a fixed sphere, which, in this case, is the sphere determined by $a$. 
Hence, we deduce the following classification of discrete R-congruences with spheres in a fixed parabolic complex:

\begin{corollary}\label{cor_parabolic_complex}
A discrete R-congruence lies in a fixed parabolic linear sphere complex if and only if its Ribaucour family contains a totally umbilic envelope.
\end{corollary}

\noindent Generically, any discrete Legendre map $f:\mathcal{V} \rightarrow \mathcal{Z}$ admits discrete R-congruences of this type: the intersection of a fixed parabolic complex $a \in \mathbb{P}(\mathcal{L})$ with the discrete Legendre map, $r_i:=f_i \cap a^\perp$, provides a Q-net in the Lie quadric $\mathbb{P}(\mathcal{L})$. 
Therefore, any discrete Legendre map admits totally umbilic Ribaucour transforms (see also \cite[Theorem 3.9]{rib_coord}).

\medskip

\subsubsection*{Elliptic complex}
Suppose that a discrete R-congruence $r: \mathcal{V} \rightarrow \mathbb{P}(\mathcal{L})$ lies in a fixed elliptic linear sphere complex $a^\perp$, $\lspan{\mathfrak{a}, \mathfrak{a}} > 0$, and fix an arbitrary R-sphere $r_{i_0}$ of the congruence. Then, there exists a 1-parameter family of contact elements $f_{i_0} \ni r_{i_0}$ with spheres that are also contained in the elliptic linear sphere complex $a^\perp$. Hence, by Lemma~\ref{transport_contact}, any envelope $f$ that is constructed from such an initial contact element $f_{i_0}$, satisfies $\lspan{\mathfrak{f}_i, \mathfrak{a}}=0$ for any $i \in \mathcal{V}$. Therefore, this envelope is spherical, that is, for any projection to a M\"obius geometry, the point sphere map lies on a fixed sphere. However, note that in this case the envelope is not totally umbilic.

Conversely, the contact elements of a spherical and not totally umbilic discrete Legendre map lie in a fixed elliptic linear complex. Thus, in summary, we have proven:

\begin{corollary}\label{cor_elliptic_complex}
A discrete R-congruence lies in a fixed elliptic linear sphere complex if and only if there exists a spherical and not totally umbilic envelope in the Ribaucour family.
\end{corollary}

\noindent We remark that the special envelopes obtained in the Corollaries \ref{cor_parabolic_complex} and \ref{cor_elliptic_complex} are Ribaucour coordinates as discussed in \cite{rib_coord}.

\subsection{Discrete isothermic R-congruences}\label{section_iso}
A discrete $Q$-net in the Lie quadric $\mathbb{P}(\mathcal{L})$ is called \emph{isothermic} if any diagonal vertex star of the $Q$-net lies in a projective 3-dimensional subspace of $\mathbb{P}(\mathbb{R}^{4,2})$ that does not contain the four outer spheres of the vertex star. Equivalently, isothermicity of a $Q$-net is characterized by the existence of a Moutard lift (see \cite{ddg_book, discrete_omega, lin_weingarten_discrete}): 

\begin{defi}
A lift $\mu: \mathcal{V} \rightarrow \mathcal{L} \subset \mathbb{R}^{4,2}$ of a discrete Q-net $s: \mathcal{V} \rightarrow \mathbb{P}(\mathcal{L})$ is called a \emph{Moutard lift} if opposite diagonals are parallel, that is,
\begin{equation*}
\delta \mu_{ik} \ || \ \delta \mu_{jl},
\end{equation*}
where $\delta \mu_{ik}:= \mu_k - \mu_i$ and $\delta \mu_{jl}:= \mu_l - \mu_j$.
\end{defi}
\noindent If the Q-net is not degenerate, hence a discrete R-congruence in the sense of this paper, the special global choice of homogeneous coordinates of a Moutard lift gives rise to additional Lie inversions that diagonally swap the R-spheres on each quadrilateral:
\begin{prop}
A lift $\mu: \mathcal{V} \rightarrow \mathcal{L} \subset \mathbb{R}^{4,2}$ of a discrete R-congruence $r: \mathcal{V} \rightarrow \mathbb{P}(\mathcal{L})$ is a Moutard lift if and only if for each quadrilateral the Lie inversion $\sigma^\delta$ with respect to the linear sphere complex $\langle \delta\mu_{ik} \rangle^\perp$ diagonally interchanges the R-spheres:
\begin{equation*}
\sigma^\delta(\mu_i)=\mu_k \ \ \text{and } \ \ \sigma^\delta(\mu_j)=\mu_l.
\end{equation*}
\end{prop}

\begin{proof}
Suppose that $\mu$ is a Moutard lift, then the linear sphere complex is determined by
\begin{equation*}
\langle \delta\mu_{ik} \rangle = \langle \delta\mu_{jl} \rangle 
\end{equation*}
and we therefore obtain
\begin{align*}
\sigma^\delta(\mu_i)=\mu_k \ \ \text{and } \ \ \sigma^\delta(\mu_j)=\mu_l.
\end{align*}

Conversely, if we have
\begin{align*}
\mu_k &= \mu_i - \frac{2\lspan{\mu_i,\delta\mu_{ik}}}{\lspan{\delta\mu_{ik},\delta\mu_{ik}}}\delta\mu_{ik},
\\\mu_l &= \mu_j - \frac{2\lspan{\mu_j,\delta\mu_{ik}}}{\lspan{\delta\mu_{ik},\delta\mu_{ik}}}\delta\mu_{ik},
\end{align*}
it follows that $\delta \mu_{ik} \ || \ \delta \mu_{jl}$ and $\mu$ is indeed a Moutard lift.
\end{proof}

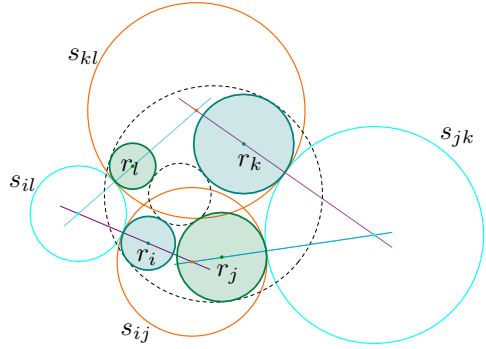
\begin{figure}[tb]\label{fig_sigma_delta}
  \begin{overpic}[width=7cm]{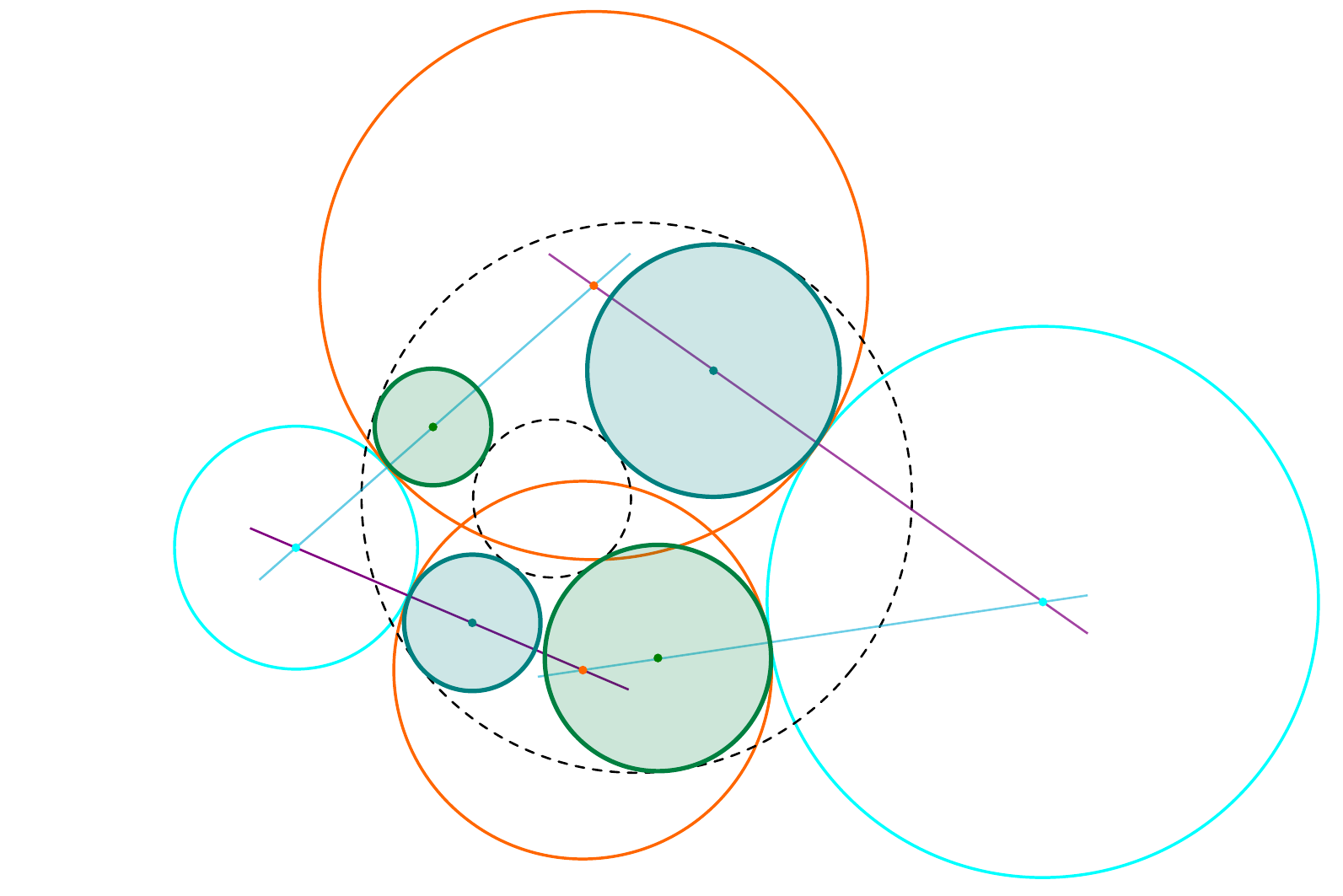}
    \put(33,17){$r_i$}
    \put(48,14){$r_j$}
    \put(52,35){$r_k$}
    \put(30,34){$r_l$}
    \put(9,31){$s_{il}$}
    \put(90,40){$s_{jk}$}
    \put(30,3){$s_{ij}$}
    \put(20,55){$s_{kl}$}
  \end{overpic}
  \caption{
    The Lie inversion $\sigma^\delta$ interchanges opposite R-spheres $r_i \leftrightarrow r_k$ and $r_j \leftrightarrow r_l$, curvature spheres $s_{ij} \leftrightarrow s_{kl}$ and $s_{il} \leftrightarrow s_{jk}$, and the corresponding contact elements $f_i \leftrightarrow f_k$ and $f_j \leftrightarrow f_l$. 
  }
  \label{fig:diagonal_inversion}
\end{figure}

\noindent The interaction of these Lie inversions $\sigma^\delta$ induced by a Moutard lift with the geometric data of the envelopes is illustrated in Figure \ref{fig_sigma_delta}:

\begin{corollary}
  The Lie inversions $\sigma^\delta$ map facewise opposite curvature spheres onto each other and diagonally interchanges the contact elements of an enveloping discrete Legendre map (see Fig.~\ref{fig:diagonal_inversion}):
\begin{align*}
\sigma^\delta(\mathfrak{s}_{ij})=\sigma^\delta(\mathfrak{s}_{kl}), \ \ \sigma^\delta(\mathfrak{s}_{jk})=\sigma^\delta(\mathfrak{s}_{il}) \ \ \text{and} \ \ 
 \sigma^\delta(\mathfrak{f}_i)=\sigma^\delta(\mathfrak{f}_{k}), \ \ \sigma^\delta(\mathfrak{f}_{j})=\sigma^\delta(\mathfrak{f}_{l}).
\end{align*}
\end{corollary}

\noindent A Moutard lift of a discrete isothermic R-congruence $r$ also defines the aforementioned Lie inversions $\one{\sigma}$ and $\two{\sigma}$: let $\mu$ denote the Moutard lift of $r$ and define the function $\lambda: \mathcal{F}\rightarrow \mathbb{R}$ by
\begin{equation*}
\lambda_{F}\, \delta \mu_{jl} = \delta \mu_{ik}
\end{equation*}
for a face $F=(ijkl)$.
Then, in terms of the homogeneous coordinates of a Moutard lift, the points $\one{n}$ and $\two{n}$ are given by
\begin{align*}
\one{\n}: &= \mu_i - \lambda_{F} \mu_j = \mu_k - \lambda_{F} \mu_l,
\\
\two{\n}: &= \mu_i + \lambda_{F} \mu_l = \mu_k + \lambda_{F} \mu_j.
\end{align*}

\noindent Moreover, since the homogeneous coordinates of a Moutard lift are global, these maps also induce an edge labelling: for any face $F = (ijkl)$, we define the edge function $e: \mathcal{E} \rightarrow \mathbb{R}$ by
\begin{align*}
\one{e}_{ij}=\one{e}_{kl}
&:=\frac{\lspan{\one{\mathfrak{n}}_F, \one{\mathfrak{n}}_F}}{2\lambda_F}= 
-\lspan{\mu_i, \mu_j} = -\lspan{\mu_k, \mu_l}
\quad \text{and} \\
\two{e}_{il}=\two{e}_{jk}
&:=-\frac{\lspan{\two{\mathfrak{n}}_F, \two{\mathfrak{n}}_F}}{2\lambda_F}=
\lspan{\mu_i, \mu_l} = \lspan{\mu_j, \mu_k}.
\end{align*}
Then, according to Lemma \ref{lem_cr_J}, the edge function $e$ factorizes the cross-ratio of the discrete isothermic R-congruence (cf.\,\cite[Lemma 3.5]{lin_weingarten_discrete}):
\begin{equation*}
cr(r_i, r_j, r_k, r_l)=
-\frac
{\lspan{\one{\mathfrak{n}}_F, \one{\mathfrak{n}}_F}}
{\lspan{\two{\mathfrak{n}}_F, \two{\mathfrak{n}}_F}}
=
\frac{e_{ij}^{(1)}}{e_{jk}^{(2)}}.
\end{equation*} 

\noindent Envelopes of discrete isothermic R-congruences belong to the special class of discrete $\Omega$-surfaces introduced in \cite{lin_weingarten_discrete}. Those are discrete Legendre maps enveloped by a pair of sphere congruences that admit K\"onigs dual lifts. In \cite{discrete_omega}, it is shown that K\"onigs dual lifts can be constructed even from one enveloping isothermic sphere congruence of a discrete Legendre map. Therefore, to summarize:

\begin{corollary}\label{cor_omega}
The Ribaucour family of a discrete isothermic R-congruence consists of discrete $\Omega$-surfaces.
\end{corollary}

\subsection{Multi-R-congruences}

If we impose the characteristic property of a discrete net not only on each quadrilateral but also on arbitrary parameter rectangles, we obtain a so-called multi-net. Various multi-nets were investigated in \cite{multinet}.

In particular, a multi-Q-net is a Q-net such that arbitrary parameter rectangles are planar. In this subsection we are interested in multi-Q-nets in the Lie-quadric:

\begin{defi}
A \emph{discrete multi-R-congruence} is a discrete R-congruence $r: \mathcal{V}\rightarrow \mathbb{P}(\mathcal{L})$ such that also the spheres of any rectangular parameter quadrilateral lie in a unique linear system.
\end{defi}

\noindent Special examples of multi-R-congruences are given by the point sphere maps of multi-circular nets; hence, geometrically, point sphere maps of discrete isothermic channel surfaces (see \cite{multinet, discrete_channel}). 

According to \cite[Theorem 2.4]{multinet}, any discrete multi-R-congruence $r:\mathcal{V} \rightarrow \mathbb{P}(\mathcal{L})$ admits a global choice of homogeneous coordinates $\mathfrak{r}:\mathcal{V} \rightarrow \mathcal{L} \subset \mathbb{R}^{4,2}$ such that we obtain  
\begin{equation}\label{equ_cond_multi_lift}
\mathfrak{r}_i - \mathfrak{r}_j + \mathfrak{r}_{j'} - \mathfrak{r}_{i'}=0
\end{equation} 
for any parameter rectangle $(ijj'i')$. Thus, the associated linear sphere complexes determined by $\one{n}$ and $\two{n}$ are constant along each coordinate ribbon and we obtain (see Fig.~\ref{fig:multi_isothermic_cs_proof} \emph{left})
\begin{corollary}
For any multi-R-congruence the Lie inversions $\one{\sigma}$ and $\two{\sigma}$ are constant along each coordinate ribbon.
\end{corollary}

\noindent As a second immediate consequence of the existence of the lift $\mathfrak{r}$, we obtain a Moutard lift for any discrete R-congruence by swapping the signs of the homogeneous coordinates $\mathfrak{r}$ along every second coordinate line in one family. Therefore, we extended the already known fact for multi-circular nets to multi-R-congruences:   

\begin{prop}
A discrete multi-R-congruence is an isothermic sphere congruence. 
\end{prop}

\noindent Therefore, by Corollary \ref{cor_omega}, the envelopes of a discrete multi-R-congruence are discrete $\Omega$-surfaces. In this special case we even obtain particular discrete $\Omega$-surfaces:  

\begin{corollary}
The Ribaucour family of a multi-R-congruence consists of discrete $\Omega$-surfaces with curvature spheres that lie in a fixed linear sphere complex along each coordinate ribbon.
\end{corollary}

\begin{figure}[tb]
  \input{F-multi_r_congruence.tikz}
  \caption{In case of a multi-R-congruence the complexes $\one{n}$ and $\two{n}$ are constant along coordinate ribbons.
    }
  \label{fig:multi_isothermic_cs_proof}
\end{figure}
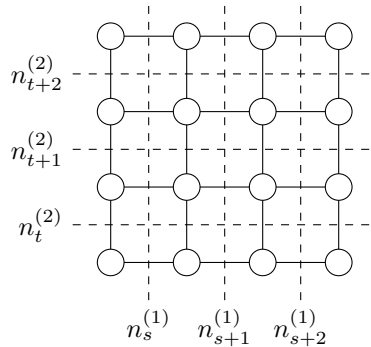

%
%

\section{Ribaucour families}\label{section_rib_family}

\subsection{Geometric properties of envelopes in the Ribaucour family} In the following paragraphs we discuss various special properties of the 2-parameter family of envelopes in the Ribaucour family of a discrete R-congruence.

\subsubsection*{Curvature spheres in a Ribaucour family}
In Lemma \ref{transport_contact} we saw that the Lie inversions $\one{\sigma}$ and $\two{\sigma}$ transport the contact elements of the envelopes along the discrete R-congruence. Since two adjacent contact elements share a common sphere, namely the curvature sphere, this sphere is fixed by the corresponding Lie inversion. Therefore, we obtain the following property for all curvature spheres in a Ribaucour family: 

\begin{prop}\label{prop_curv_spheres}
The curvature spheres of opposite edges of any envelope in a Ribaucour family  lie in the same linear sphere complex.
\end{prop}

\subsubsection*{Totally umbilic faces} If the four curvature spheres of a face of a discrete Legendre map coincide, the face is called \textit{umbilic}. While in the smooth case umbilic points are rather special, in any $(2,1)$-Ribaucour family we obtain envelopes with umbilic faces:

\begin{lem}\label{lem_totally_umbilic}
Let $f$ be a discrete Legendre map in a $(2,1)$-Ribaucour family. A face $(ijkl)$ of $f$ is umbilic if and only if the four contact elements $f_i, f_j, f_k$ and $f_l$ are also contact elements of the Dupin cyclide generated by the R-spheres of this face.

In this case, for any projection to a M\"obius geometry, the four point spheres of $f$ lie on a curvature circle of the Dupin cyclide.
\end{lem}

\begin{proof}
Let $f$ be a discrete Legendre map in a $(2,1)$-Ribaucour family and assume that the contact elements $f_i, f_j, f_k$ and $f_l$ of a face $(ijkl)$ coincide with contact elements of the Dupin cyclide $d= D_1 \oplus_\perp D_2$, where the R-spheres $r_i, r_j, r_k$ and $r_l$ lie in the $(2,1)$-plane $D_1$. Then, $f_i=\spann{r_i, s_i}$, where $s_i \in D_2$. But, since $s_i$ is also in oriented contact with the R-spheres $r_j, r_k$ and $r_l$, the sphere $s_i$ is the constant curvature sphere of the envelope $f$ at this face. Hence this face of $f$ is umbilic.

Conversely, suppose that an envelope $f$ of a discrete $(2,1)$-R-congruence has a totally umbilic face, then the constant curvature sphere has to be in oriented contact with all four R-spheres of this face. Since these four R-spheres determine a Dupin cyclide, the constant curvature sphere of $f$ is also a curvature sphere of this Dupin cyclide (in the other curvature sphere family than the R-spheres). This completes the proof.
\end{proof}

\noindent Recall that, by Corollary \ref{cor_initial_contact_el}, any choice of an initial contact element at an initial vertex provides a unique envelope in the Ribaucour family. Therefore, with the help of Lemma \ref{lem_totally_umbilic}, we can construct umbilic faces at any face of a discrete $(2,1)$-R-congruence:  
 
\begin{corollary}
For any face of a discrete $(2,1)$-R-congruence there exists a 1-parameter family of envelopes in the Ribaucour family that are umbilic at this face.
\end{corollary}

\subsubsection*{Permutability theorems} A key property in the theory of transformations is the existence of permutability theorems: given two Ribaucour transforms $f_1$ and $f_2$ of a Legendre map $f$, then there exists a 1-parameter family of Legendre maps that are simultaneous transforms of $f_1$ and $f_2$. Moreover, corresponding points of these Legendre maps are concircular. This result, generically, holds for smooth, as well as, for discrete Legendre maps (cf.\,\cite[Theorem 3.6]{ddg_book}, \cite{MR2254053}, \cite[\S 8]{MR2004958}).  

However, contrary to the smooth case, for any discrete Legendre map circularity of corresponding points in the permutability theorem can fail: four envelopes in a Ribaucour family obviously satisfy the permutability theorem. But, if we project to a M\"obius geometry, corresponding point spheres of the four envelopes do not have to lie on a circle. The point spheres lie on the corresponding R-sphere and are therefore in general only cospherical.

\begin{prop}\label{prop_permutability_unusual}
Let $f$ be a discrete Legendre map, then there exist three Ribaucour transforms $f_1$, $f_2$ and $f_{12}$ of $f$ such that $f_{12}$ is a simultaneous Ribaucour transform of $f_1$ and $f_2$ and, for any projection to a M\"obius geometry, corresponding point spheres of these four nets are not circular.
\end{prop}

We emphasize that if corresponding contact elements of $f$, $f_1$, and $f_2$ span a 3-dimensional projective subspace of signature $(2,2)$, then there only exists a 1-parameter family of simultaneous Ribaucour transforms $f_{12}$. The points of the contact elements will then be circular. Hence in this case, the usual permutability theorem holds.

We remark that the significant difference to the smooth transformation theory pointed out in Proposition \ref{prop_permutability_unusual} leads to examples of discrete nets that have no counterparts in the smooth surface theory (see for example \cite[\S 1.1]{discrete_channel}).

\subsubsection*{Deformations of envelopes in the Ribaucour family} The ambiguity of envelopes in the Ribaucour family of a discrete R-congruence $r:\mathcal{V} \rightarrow \mathbb{P}(\mathcal{L})$ provides a possibility to smoothly deform two envelopes $f$ and $\hat{f}$ into each other.

Let $r_0$ be an R-sphere of the congruence $r$ at an initial vertex $v_0 \in \mathcal{V}$. Then, by Corollary~\ref{cor_initial_contact_el}, any choice of a smooth initial Legendre curve 
\begin{equation*}
\gamma_0:[0,1] \rightarrow \mathcal{Z} \ \ \text{with }
\begin{cases}
r_0 \in \gamma_0(t) \ \text{for any } t \in [0,1], 
\\[3pt]\gamma_0(0)=f_0 \ \ \text{and} \ \ \gamma_0(1)=\hat{f}_0
\end{cases}
\end{equation*}
gives rise to a 1-parameter family $\{f^{\gamma_0(t)}\}_{t \in [0,1] }$ of envelopes in the Ribaucour family (see Fig.~\ref{fig:ribaucour_curve}).  

\begin{figure}[bt]
\centering
\includegraphics[width=10cm]{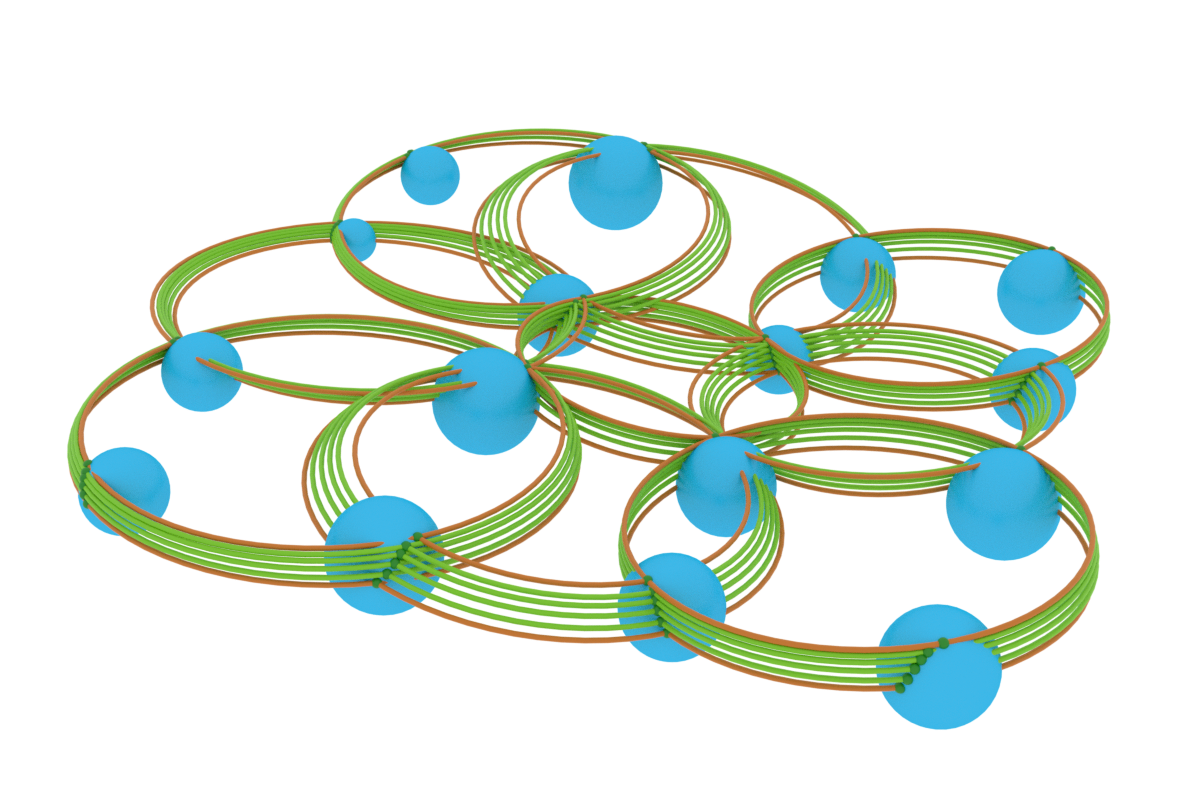}
\caption{Two envelopes $f$ and $\hat{f}$ of the Ribaucour family (orange) can be interpolated with a 1-parameter family of envelopes (green) that are constructed from an initial curve~$\gamma$ going through two initial contact elements of $f$ and $\hat{f}$.
}
\label{fig:ribaucour_curve}
\end{figure}
%
In particular, this construction yields a smooth Legendre curve $t \mapsto f^{\gamma_0(t)}_{v}$ for each vertex $v \in \mathcal{V}$ that lies in the fixed parabolic linear sphere complex determined by the corresponding R-sphere $r_v$. Moreover, according to Lemma \ref{transport_contact}, two adjacent Legendre curves are related by a Lie inversion. Thus, they envelop a 1-parameter family of spheres and form a smooth Ribaucour pair of curves.

\medskip

\noindent If projecting to a M\"obius geometry, the point sphere envelopes of $\{f^{\gamma_0(t)}\}_{t \in [0,1] }$ provide smooth deformations of circular nets, where the vertices move along spherical curves. In particular, if the initial curve $\gamma_0$ is a (part of a) circle, then circularity of it is preserved along the entire discrete R-congruence (cf.\,Proposition \ref{prop_moebius_inversion_edge}).

\medskip

\noindent The construction pointed out in this section provides a way of obtaining discrete and semi-discrete triply orthogonal systems with special vertical coordinate surfaces having one family of spherical curvature lines. For particular choices of the initial curve $\gamma_0$ we even obtain semi-discrete cyclic systems. A deeper analysis of these systems will be given in a future work.

\subsection{Discrete Ribaucour pairs in the Ribaucour family}\label{subsec_rib_pairs} 

To gain further geometric insights into the Ribaucour family of a discrete R-congruence, in this subsection, we fix two envelopes of a discrete R-congruence, classically called a \emph{Ribaucour pair} of discrete Legendre maps.
The contact elements of a Ribaucour pair form a fundamental line system in the sense of~\cite{bobenko_schief-discrete_line_complexes, DoliwaSantiniManas:2000:TransformationsOfQnets} since R-congruences are Q-nets.

\medskip

\noindent Thus, suppose that $f, \hat{f}:\mathcal{V}\rightarrow \mathcal{Z}$ are two discrete Legendre maps enveloping the discrete R-congruence $r:\mathcal{V}\rightarrow \mathbb{P}(\mathcal{L})$. Then, for any edge $(ij)$ the contact elements $(f_i, f_j, \hat{f}_j, \hat{f}_i)$ provide a quadrilateral of a discrete Legendre map, that is, adjacent contact elements share a common sphere. Thus, along each coordinate line of a Ribaucour pair the contact elements of $f$ and $\hat{f}$ yield a coordinate ribbon of a discrete Legendre map; these coordinate ribbons will be called \emph{vertical ribbons of the Ribaucour pair}. The ``curvature spheres'' of the vertical ribbons are given by the R-spheres and the curvature spheres of $f$ and $\hat{f}$ along the corresponding coordinate line of the Ribaucour pair.
Moreover, note that any two consecutive vertical ribbons also form a Ribaucour pair enveloping curvature spheres of $f$ and $\hat{f}$. 

In the following paragraphs we demonstrate how the structures of the various Ribaucour pairs interact with each other.

\begin{prop}\label{prop_three_n}
Let $f, \hat{f}:\mathcal{V}\rightarrow \mathcal{Z}$ be two envelopes of  a discrete R-congruence. Then there exists a map 
\begin{equation*}
\three{n}:\mathcal{F}\rightarrow \mathbb{P}(\mathbb{R}^{4,2}) \setminus \mathbb{P}(\mathcal{L})
\end{equation*}
such that for any quadrilateral $(ijkl)$ the induced Lie inversion $\three{\sigma}:=\sigma_{\three{\mathfrak{n}}}$ interchanges the curvature spheres of $f$ and $\hat{f}$,
\begin{align*}
\sigma_{\three{\mathfrak{n}}}(\mathfrak{s}_{ij})=\hat{\mathfrak{s}}_{ij}, \ 
\sigma_{\three{\mathfrak{n}}}(\mathfrak{s}_{jk})=\hat{\mathfrak{s}}_{jk}, \ 
\sigma_{\three{\mathfrak{n}}}(\mathfrak{s}_{kl})=\hat{\mathfrak{s}}_{kl}, \ 
\sigma_{\three{\mathfrak{n}}}(\mathfrak{s}_{il})=\hat{\mathfrak{s}}_{il}.
\end{align*}
Moreover, for any quadrilateral the Lie inversion $\three{\sigma}$ preserves the R-spheres, maps corresponding contact elements of the Ribaucour pair $(f, \hat{f})$ onto each other and is involutive to $\one{\sigma}$ and $\two{\sigma}$ (see Fig.~\ref{fig:sigma3} for notation).
\end{prop}
\begin{proof}
Firstly observe that, since $\mathfrak{s}_{kl}=\two{\sigma}(\mathfrak{s}_{ij})$ and $\hat{\mathfrak{s}}_{kl}=\two{\sigma}(\hat{\mathfrak{s}}_{ij})$, we can choose homogeneous coordinates such that
\begin{align*}
0=\mathfrak{s}_{ij} - \hat{\mathfrak{s}}_{ij} - \mathfrak{s}_{kl} +\hat{\mathfrak{s}}_{kl},
\end{align*}
as well as homogeneous coordinates
\begin{equation*}
0=\mathfrak{s}_{jk} - \hat{\mathfrak{s}}_{jk} - \mathfrak{s}_{il} +\hat{\mathfrak{s}}_{il}
\end{equation*}
for the other four curvature spheres. Moreover, we define two vectors $n^{(3)}$ and $\tilde{n}^{(3)}$ by
\begin{align*}
\mathfrak{n}^{(3)}&:=\mathfrak{s}_{ij} - \hat{\mathfrak{s}}_{ij} = \mathfrak{s}_{kl} - \hat{\mathfrak{s}}_{kl},
\\\tilde{\mathfrak{n}}^{(3)}&:= \mathfrak{s}_{il} - \hat{\mathfrak{s}}_{il}=\mathfrak{s}_{jk} - \hat{\mathfrak{s}}_{jk}.
\end{align*}
Then, the R-spheres $r_i, r_j, r_k$ and $r_l$ lie in the induced linear sphere complexes $n^{(3)}$ and $\tilde{n}^{(3)}$, which are therefore involutive to $\one{n}$ and $\two{n}$. 

Furthermore, since 
\begin{align*}
\mathfrak{s}_{il}&= \lspan{\mathfrak{s}_{ij}, \two{\mathfrak{n}}} \mathfrak{r}_i - \lspan{\mathfrak{r}_{i}, \two{\mathfrak{n}}} \mathfrak{s}_{ij},
\\\hat{\mathfrak{s}}_{il}&= \lspan{\hat{\mathfrak{s}}_{ij}, \two{\mathfrak{n}}} \mathfrak{r}_i - \lspan{\mathfrak{r}_{i}, \two{\mathfrak{n}}} \hat{\mathfrak{s}}_{ij},
\end{align*}
the two linear sphere complexes induced by $n^{(3)}$ and $\tilde{n}^{(3)}$ coincide:
\begin{align*}
\tilde{\mathfrak{n}}^{(3)}&=\lspan{\two{\mathfrak{n}},\three{\mathfrak{n}}}\mathfrak{r}_i - \lspan{\two{\mathfrak{n}},\mathfrak{r}_i}\three{\mathfrak{n}}=- \lspan{\two{\mathfrak{n}},\mathfrak{r}_i}\three{\mathfrak{n}}.
\end{align*}
\end{proof}

\begin{figure}[tb]
  \centering
  \input{H-sigma3.tikz}
  \caption{Two envelopes $f$ and $\hat{f}$ of a discrete R-congruence define a linear sphere complex
    $\three{n}$ and the corresponding Lie inversion~$\three{\sigma}$. The
    Lie inversion interchanges curvature spheres and contact elements of the two
    envelopes while fixing the R-spheres of the congruence.  
  }
  \label{fig:sigma3}
\end{figure}
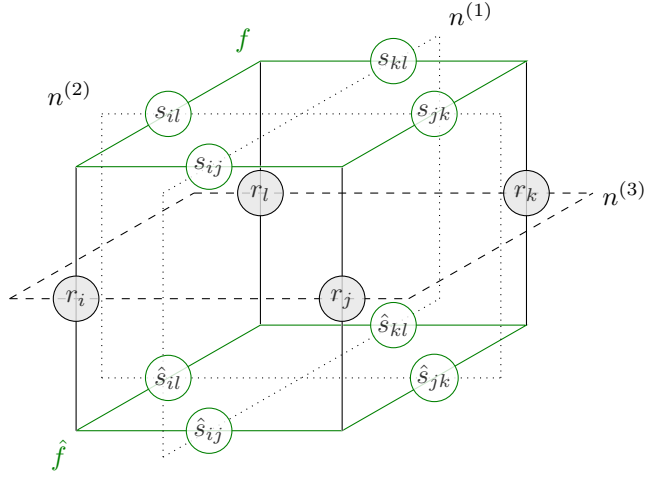

Thus, on each hexahedron of a Ribaucour pair $(f, \hat{f})$ we obtain the following symmetric configuration of the curvature spheres and the R-spheres:
\begin{equation}\label{symmetric_configuration}
\begin{aligned}
\one{s}_{ij}, \one{s}_{kl}, \one{\hat{s}}_{kl}, \one{\hat{s}}_{ij} \perp  \one{n},
\\\two{s}_{jk}, \two{s}_{il}, \two{\hat{s}}_{il}, \two{\hat{s}}_{jk} \perp \two{n},
\\r_i, r_j, r_k, r_l \perp \three{n}.
\end{aligned}
\end{equation}
In particular, for $\nu=1,2,3$, the Lie inversion $\sigma^{(\nu)}$ fixes the four spheres assigned to the $(\nu)$-edges and interchanges the spheres within the other two linear sphere complexes.

Thus, the Lie inversions $\three{\sigma}$ given in Proposition \ref{prop_three_n} reveal a crucial property between two envelopes of a discrete R-congruence:

\begin{theorem}
Two envelopes of a discrete R-congruence are related by a facewise constant Lie inversion.
\end{theorem}

Furthermore, as a consequence of Lemma \ref{lem_inversions_commute}, the Lie inversions $\three{\sigma}$ constructed in Proposition~\ref{prop_three_n} induce pairings in the entire Ribaucour family:

\begin{corollary}\label{cor_rib_pairs}
Let $(f, \hat{f})$ be a discrete Ribaucour pair of a discrete R-congruence. Then the corresponding Lie inversions $\three{\sigma}$ decompose the Ribaucour family into discrete Ribaucour pairs.
\end{corollary}

\noindent We emphasize that the Lie inversions $\three{\sigma}$ and the induced decomposition of the Ribaucour family depend on the initial choice of the Ribaucour pair $(f, \hat{f})$. Hence, the decomposition given in Corollary \ref{cor_rib_pairs} is not unique.

\medskip

\noindent The Ribaucour family of a discrete R-congruence in a fixed elliptic linear sphere complex $a^\perp$, $\lspan{\mathfrak{a}, \mathfrak{a}} > 0$, as discussed in the Subsection \ref{subsection_fixed_complex} admits a special decomposition into Ribaucour pairs: let $f:\mathcal{V}\rightarrow \mathcal{Z}$ be an envelope that is not spherical, then the Legendre map $\hat{\mathfrak{f}}_i:=\sigma_\mathfrak{a}(\mathfrak{f}_i)$ also lies in the Ribaucour family. Hence, we obtain Ribaucour pairs related by the constant Lie inversion $\sigma_a$. 

This fact coincides with the smooth case, where a smooth Ribaucour pair enveloping a sphere congruence with a constant osculating complex is also related by a fixed Lie inversion (\cite[\S 89]{blaschke}).

\subsection{Cyclidic nets in the Ribaucour family}\label{subsec_cyclidic_nets}
For any face of a discrete Legendre map there exists a 1-parameter family of \emph{face-cyclides}, Dupin cyclides that share the four curvature spheres assigned to a face with the discrete Legendre map (cf.\,\cite{paper_cyclidic, discrete_channel}). Note that any face-cyclide has four distinguished curvature lines, namely the curvature lines that join two adjacent contact elements of the discrete Legendre map and lie on the corresponding curvature sphere.
We will denote the space of $(2,1)$-planes in $\mathbb{R}^{4,2}$ by $G_{(2,1)}(\mathbb{R}^{4,2})$ and Dupin cyclides by two complementary $(2,1)$-planes $\one{D}$ and $\two{D}$ with $\one{D} \oplus_\perp \two{D} = \mathbb{R}^{4,2}$.

\medskip

To begin with, we shall point out how the Lie inversions $\one{\sigma}$ and $\two{\sigma}$ corresponding to a discrete R-congruence interact with the face-cyclides of its envelopes.

\begin{lem}
A congruence of face-cyclides of an envelope is preserved by the Lie inversions $\one{\sigma}$ and $\two{\sigma}$; however, the contact elements of opposite curvature lines going through the vertices of the discrete Legendre map are interchanged by the corresponding Lie inversion. 
\end{lem}
\begin{proof}
Let $d=\one{D} \oplus_\perp \two{D}$ be a face-cyclide of the face $(ijkl)$ with $s_{ij}, s_{kl} \in \one{D}$ and $s_{il}, s_{jk} \in \two{D}$. Moreover, let $\tilde{s} \in \one{D}$ be a curvature sphere of the face-cyclide. Then $\lspan{\tilde{\mathfrak{s}}, \mathfrak{s}_{il}}=\lspan{\tilde{\mathfrak{s}}, \mathfrak{s}_{jk}}=0$ and we conclude that $\tilde{s}\perp \one{n}$. Thus, the spheres in $\one{D}$ are fixed by the Lie inversion $\one{\sigma}$ and, therefore, also the face-cyclide as unparametrized surface.

Furthermore, any contact element along the curvature line going through $f_i$ and $f_l$ is given by $\lspan{s_{il},\tilde{s}}$ for a sphere $\tilde{s}\in \one{D}$. Since 
\begin{equation*}
\one{\sigma}(\mathfrak{s}_{il})= \mathfrak{s}_{jk} \ \ \text{and } \ \one{\sigma}(\tilde{\mathfrak{s}})=\tilde{\mathfrak{s}},
\end{equation*}
these contact elements are mapped to the contact elements of the opposite curvature line of the face-cyclide. Analogous arguments for the other pair of curvature lines complete the proof.
\end{proof}

This lemma also provides a construction for a special congruence 
\begin{equation*}
d:\mathcal{F}\rightarrow G_{(2,1)}(\mathbb{R}^{4,2}) \times G_{(2,1)}(\mathbb{R}^{4,2})
\end{equation*}
of face-cyclides for $f$ from a given face-cyclide $d_{\alpha}=\one{D}_{\alpha} \oplus_{\perp} \two{D}_{\alpha}$ of an initial face~$(ijj'i')$ (notations see Fig.\,\ref{fig_labels}): let  
\begin{equation}\label{equ_d_beta}
\one{D}_{\beta}:= (s_{jj'} \oplus \one{D}_{\alpha}) \cap (\one{n}_{\beta})^{\perp} \ \ \text{and } \ \two{D}_\beta:=(\one{D}_{\beta})^\perp,
\end{equation}
then $d_\beta= \one{D}_{\beta} \oplus \two{D}_\beta$ is a face-cyclide for the face~$(jkk'j')$. Furthermore, defining 
\begin{equation}\label{equ_d}
\begin{aligned}
\two{D}_\gamma&:= (s_{j'k'} \oplus \two{D}_\beta) \cap (\two{n}_{\gamma})^{\perp},
\\\one{D}_\delta&:=(s_{j'j''} \oplus \one{D}_\gamma) \cap (\one{n}_{\delta})^{\perp},
\end{aligned}
\end{equation}
yields unique face-cyclides $d_\gamma$ and $d_\delta$ for the four faces of the vertex-star. This construction consistently extends on all faces and provides a face-cyclide congruence for the discrete Legendre map.

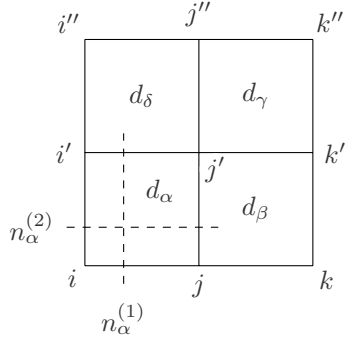
\begin{figure}[tb]
  \input{notation.tikz}
  \caption{
    Notation used for face cyclides on four quadrilaterals at a common vertex.
  }
  \label{fig_labels}
\end{figure}

Projecting to any M\"obius geometry $\langle \mathfrak{p} \rangle^\perp$, reveals a remarkable property of the just constructed face-cyclide congruence $d$. 

Firstly, observe that the face-cyclides $d_\alpha$ and $d_\beta$ share a common curvature line $c_{\alpha\beta}$, namely the circle of point spheres joining the point spheres $p_j \in f_j$ and $p_{j'} \in f_{j'}$. The contact elements along $c_{\alpha\beta}$ coincide due to Equation \eqref{equ_d_beta}, and hence the two face-cyclides define a piecewise smooth surface that is $C^1$ across the common curvature line. 

In the same way the pairs $\{d_\beta, d_\gamma\}$ and $\{d_\gamma, d_\delta\}$ meet at a common circular curvature line $c_{\beta\gamma}$ and $c_{\gamma\delta}$, respectively. Note that the circles $c_{\alpha\beta}$ and $c_{\beta\gamma}$, as well as $c_{\beta\gamma}$ and $c_{\gamma\delta}$, intersect orthogonally. Therefore, since the circular curvature lines $c_{\alpha\beta}, c_{\beta\gamma}$ and $c_{\gamma\delta}$ of the face-cyclides around the vertex-star go through the common contact element $f_{j'}$, also the face-cyclides $d_\delta$ and $d_\alpha$ meet at a common curvature line $c_{\delta\alpha}$. Since all contact elements of $d_\delta$ and $d_\alpha$ along the common circular curvature line $c_{\delta\alpha}$ contain the curvature sphere $s_{i'j'}$ of $f$, also the contact elements of $d_\delta$ and $d_\alpha$ along $c_{\delta\alpha}$ coincide. 

Thus, with the help of the Lie inversions $\one{\sigma}$ and $\two{\sigma}$, we have constructed a particular face-cyclide congruence, discussed in \cite{cas} and \cite{paper_cyclidic}:
%
%
\begin{defi}
Let $f: \mathcal{V} \rightarrow \mathcal{Z}$ be a discrete Legendre map, then a \emph{cyclidic net of $f$} is a congruence of face-cyclides $d: \mathcal{F} \rightarrow G_{(2,1)}(\mathbb{R}^{4,2}) \times G_{(2,1)}(\mathbb{R}^{4,2})$ such that two face-cyclides adjacent along the edge $(ij)$ share the same contact elements along the curvature direction going through $f_i$ and $f_j$.
\end{defi}

Note that the construction given above only depends on the choice of an initial face-cyclide and is independent of the discrete R-congruence $r$; Lie inversions $\one{\sigma}$ and $\two{\sigma}$ determined by any other  discrete R-congruence of $f$ lead to the same cyclidic net for $f$. 

In particular, this construction gives a cyclidic net for a principal net in a conformal geometry $\lspan{\mathfrak{p}}^\perp$: since its point sphere congruence is a special discrete R-congruence, the corresponding Lie inversions $\one{\sigma}$ and $\two{\sigma}$ descend to M\"obius transformations and can be used to determine the Dupin cyclide patches of a cyclidic net (cf.\,(\ref{equ_d_beta}) and (\ref{equ_d})).
\medskip 

The symmetries described in (\ref{symmetric_configuration}) can be exploited to relate cyclidic nets of a Ribaucour pair: 

\begin{theorem}\label{thm_cyclidic_for_Rib}
Let $(f, \hat{f})$ be a discrete Ribaucour pair related by the facewise constant Lie inversion $\three{\sigma}$. If $d=(\one{D}, \two{D}):\mathcal{F}\rightarrow G_{(2,1)}(\mathbb{R}^{4,2}) \times G_{(2,1)}(\mathbb{R}^{4,2})$ is a cyclidic net of $f$, then 
\begin{equation*}
\hat{d}:\mathcal{F} \rightarrow G_{(2,1)}(\mathbb{R}^{4,2}) \times G_{(2,1)}(\mathbb{R}^{4,2}), \ \ \hat{d}=(\three{\sigma}(\one{D}), \ \three{\sigma}(\two{D}) )
\end{equation*}
provides a cyclidic net for $\hat{f}$.
\end{theorem}

\begin{proof}
Since corresponding curvature spheres of $f$ and $\hat{f}$ are mapped onto each other by the Lie inversions $\three{\sigma}$ 
and Dupin cyclides are invariant under Lie inversions, $\hat{d}_{ijkl}$ defines a face-cyclide for the face $(ijkl)$ of $\hat{f}$.

Moreover, to prove that $\hat{d}$ indeed provides a cyclidic net of $\hat{f}$, we consider two adjacent faces along an edge $(ij)$: firstly, observe that the Lie inversions $\three{\sigma}_{n}$ and $\three{\sigma}_{\bar{n}}$ belonging to the two adjacent faces are determined by the two linear sphere complexes
\begin{equation*}
\three{\mathfrak{n}}:=\mathfrak{s}_{ij}-\lambda \hat{\mathfrak{s}}_{ij} \ \text{and } \ \three{\bar{\mathfrak{n}}}:=\mathfrak{s}_{ij}-\bar{\lambda} \hat{\mathfrak{s}}_{ij},
\end{equation*} 
where $\mathfrak{s}_{ij} \in s_{ij}$, $\hat{\mathfrak{s}}_{ij} \in \hat{s}_{ij}$ and $\lambda, \bar{\lambda} \in \mathbb{R}$ are appropriately chosen. Thus, by Lemma~\ref{lem_inversion_family}~\ref{lem_inversion_two}, the contact elements of the face-cyclides of $f$ along the common curvature line passing through $f_i$ and $f_j$ are mapped to the same contact elements by $\three{\sigma}_{n}$ and $\three{\sigma}_{\bar{n}}$. Thus, two adjacent face-cyclides of $\hat{d}$ share common contact elements along the curvature line through $\hat{f}_i$ and $\hat{f}_j$. 
\end{proof}


\noindent In \cite[Definition 4.4]{trafo_channel}, the existence of two special Dupin cyclide congruences for a smooth Ribaucour pair of Legendre maps was pointed out. We report on a similar construction in the discrete case: 

\begin{defi}\label{def_rib_cyclides}
Let $f, \hat{f}:\mathcal{V}\rightarrow \mathcal{Z}$ be two envelopes of  a discrete R-congruence. Face-cyclides along a vertical ribbon will be called \emph{R-cyclides of the Ribaucour pair $(f, \hat{f})$}, that is, for a vertical face, a Dupin cyclide $\delta=R \oplus_\perp \tilde{R} \subseteq \mathbb{R}^{4,2}$ satisfying
\begin{equation*}
\mathfrak{s}_{ij}, \hat{\mathfrak{s}}_{ij} \in R \ \ \text{and} \ \  \mathfrak{r}_i, \mathfrak{r}_j \in \tilde{R}.
\end{equation*}
\end{defi}

In Theorem \ref{thm_cyclidic_for_Rib}, we have learned that cyclidic nets for a Ribaucour pair arise in distinguished pairs, where the face-cyclides are related by the Lie inversions $\three{\sigma}$. For these cyclidic nets there exists a canonical choice for the R-cyclides on the vertical ribbons:

\begin{corollary}\label{cor_induced_R_cyclides}
Suppose that $d$ and $\hat{d}$ are cyclidic nets of a Ribaucour pair $(f, \hat{f})$ related by the Lie inversions $\three{\sigma}$. Then for an edge $(ij)$ of the Ribaucour pair, the contact elements along two corresponding curvature lines of $d$ and $\hat{d}$ passing through $f_i$ and $f_j$, as well as $\hat{f}_i$ and $\hat{f}_j$, uniquely determine an R-cyclide for the corresponding vertical face.
\end{corollary}
\begin{proof}
Since the contact elements along the curvature lines under consideration are related by the Lie inversion $\three{\sigma}$, two corresponding contact elements share a common sphere lying in $(\three{n})^\perp$. In this way, we obtain a 1-parameter family of spheres that are in oriented contact with the spheres $s_{ij}$ and $\hat{s}_{ij}$ and are therefore curvature spheres of a face-cyclide for the vertical face.  
\end{proof}

We deduce that, by construction, the induced R-cyclides investigated in Corollary \ref{cor_induced_R_cyclides} provide (one ribbon of) a cyclidic net along each vertical ribbon. However, observe that two adjacent induced R-cyclides belonging to two vertical ribbons from different coordinate directions do not share a common curvature line. In particular, these Dupin cyclides do not give a 3D cyclidic net as introduced in \cite[Section 3.2]{paper_cyclidic}.

\medskip

\section{Envelopes with spherical curvature lines}\label{section_spherical}

In this section, we will draw attention to the Ribaucour transformation of discrete channel surfaces as discussed in \cite{discrete_channel} and the wider class of discrete Legendre maps with a family of spherical curvature lines.

\subsection{Discrete spherical curvature lines}
Inspired by the classification of spherical curvature lines in the smooth case (see \cite{blaschke, spherical_curv_lines}), we introduce the notion of osculating complexes for discrete Legendre maps.
To obtain uniqueness of the osculating complexes, we suppose a mild genericity condition on the discrete Legendre map. Note that also in the smooth case, uniqueness fails for the class of channel surfaces. 

\medskip

Thus, let $f: \mathcal{V} \rightarrow \mathcal{Z}$ be a discrete Legendre map and fix a point sphere complex $\mathfrak{p}\in \mathbb{R}^{4,2}$, $\lspan{\mathfrak{p},\mathfrak{p}} < 0$. Furthermore, suppose that four consecutive contact elements $f_{i'}, f_i, f_j$ and $f_{j'}$ along a coordinate line are nowhere circular, that is, the four point spheres $p_{i'}, p_i, p_j$ and $p_{j'}$ do not lie on a circle. 

Then the spheres of these four contact elements lie in a unique elliptic linear sphere complex: let $s$ and $\tilde{s}$ be the two oriented spheres that contain the four point spheres $p_{i'}, p_i, p_j$ and $p_{j'}$. Then the sought-after linear sphere complex is given by
\begin{equation*}
\mathfrak{t}:={\lspan{\tilde{\mathfrak{s}}, \mathfrak{s}_{ij}}}\mathfrak{s} - {\lspan{\mathfrak{s}, \mathfrak{s}_{ij}}}\tilde{\mathfrak{s}},
\end{equation*}
where $s_{ij}$ denotes the curvature sphere belonging to the edge $(ij)$. Clearly, the spheres of the contact elements $\mathfrak{f}_i:=\spann{\mathfrak{s}_{ij}, \mathfrak{p}_i}$ and $\mathfrak{f}_j:=\spann{\mathfrak{s}_{ij}, \mathfrak{p}_j}$ lie in $t^\perp$. Moreover, since $s_{ii'}\in f_i$ and $s_{jj'}\in f_j$, also the spheres of the contact elements
\begin{equation*}
f_{i'}=\spann{s_{ii'}, p_i} \ \ \text{and} \ \ f_{j'}=\spann{s_{jj'}, p_j}
\end{equation*}
are contained in $t^\perp$.

\begin{defi}
Let $f:\mathcal{V} \rightarrow \mathcal{Z}$ be a nowhere circular discrete Legendre map, then 
\begin{equation*}
t:\mathcal{E}\rightarrow \mathbb{P}(\mathbb{R}^{4,2}), \ (ij) \mapsto t_{ij}:= \spann{ f_{i'}, f_{i}, f_{j} ,f_{j'}}^\perp
\end{equation*}
are called the \emph{osculating complexes of $f$}.
\end{defi}

\noindent As an immediate consequence of the above considerations, we can characterize spherical curvature lines of a discrete Legendre map:
\begin{prop}\label{spherical_osculating}
The $(1)$-coordinate lines of a nowhere circular discrete Legendre map are spherical if and only if the osculating complexes along each $(1)$-coordinate line are constant. 
\end{prop}

\subsection{Ribaucour transformations of discrete channel surfaces}
Curvature spheres of discrete channel surfaces are constant in the circular direction and hence the corresponding contact elements along the circular parameter lines lie in a parabolic linear sphere complex.
We observe the following property if the spheres of an R-congruence lie in a parabolic subcomplex:

\begin{prop}\label{prop_const_curvsphere_envelope}
A discrete R-congruence admits an envelope with constant curvature spheres along each $(1)$-coordinate line if and only if along each $(1)$-coordinate line the R-spheres lie in a parabolic complex and the map $n^{(2)}$ is constant along each $(1)$-coordinate ribbon.
\end{prop}

\begin{proof}
Suppose that $r:\mathcal{V}\rightarrow \mathbb{P}(\mathcal{L})$ is a discrete R-congruence admitting an envelope with a constant curvature sphere along each $(1)$-coordinate line. Then, along each such coordinate line the constant curvature sphere and the R-spheres are in oriented contact and therefore lie in a fixed parabolic linear sphere complex. 

Furthermore, let us consider a $(1)$-coordinate ribbon bounded by two $(1)$-coordinate lines $\gamma_{i}$ and $\gamma_{i+1}$.  Moreover, without loss of generality, we choose homogeneous coordinates such that the induced Lie inversions $\two{\sigma}_\alpha$ and $\two{\sigma}_\beta$ of two faces adjacent along the edge $(jk)$ of the coordinate ribbon are given by 
\begin{equation*}
\mathfrak{n}_\alpha^{(2)}:= \mathfrak{r}_j - \mathfrak{r}_k \ \ \text{and } \ \ \mathfrak{n}_\beta^{(2)}:= \mathfrak{r}_j - \lambda_\beta \mathfrak{r}_k,
\end{equation*}
where $\lambda_\beta \in \mathbb{R}\setminus \{ 0 \}$ is a  suitable constant.

Then, denoting the constant curvature spheres along the coordinate line $\gamma_{i}$ by $\one{s}_{i}$, we conclude that the spheres given by $\sigma_{\mathfrak{n}_\alpha^{(2)}}(\one{\mathfrak{s}}_i)$ and $\sigma_{\mathfrak{n}_\beta^{(2)}}(\one{\mathfrak{s}}_i)$ have to coincide. Hence, there exists a constant $c \in \mathbb{R}\setminus \{ 0\}$ such that
\begin{align*}
0&=\sigma_{\mathfrak{n}_\alpha^{(2)}}(\one{\mathfrak{s}}_{i})- c \sigma_{\mathfrak{n}_\beta^{(2)}}(\one{\mathfrak{s}}_{i})
\\&= (1-c)\one{\mathfrak{s}}_{i} + (1-c \lambda_\beta)\frac{\lspan{\one{\mathfrak{s}}_{i},\mathfrak{r}_k}}{\lspan{\mathfrak{r}_k,\mathfrak{r}_j}}  \mathfrak{r}_k + (c-1)\frac{\lspan{\one{\mathfrak{s}}_{i},\mathfrak{r}_k}}{\lspan{\mathfrak{r}_k,\mathfrak{r}_j}} \mathfrak{r}_j.
\end{align*}
Therefore, since $r_j, r_k$ and $\one{s}_{i}$ are linearly independent, each scalar factor has to vanish and we conclude that $n^{(2)}_\alpha = \two{n}_\beta$. So it is constant along each $(1)$-coordinate ribbon. Geometrically, the constant $\two{n}$ is the intersection of the lines $\lspan{\one{s}_i, \one{s}_{i+1}}$ and $\lspan{r_j, r_k}$.

Conversely, assume that along each $(1)$-coordinate line the R-spheres lie in a fixed parabolic complex. Then, in particular, along each $(1)$-coordinate line all R-spheres are in oriented contact with the constant curvature sphere along this coordinate line. Furthermore, since the map $n^{(2)}$ is constant along each $(1)$-coordinate ribbon, the choice of an initial contact element containing the corresponding constant curvature sphere reveals the sought-after envelope of the R-congruence (cf.\,Lemma \ref{transport_contact}).  
\end{proof}

\noindent We recall that a discrete Legendre map is a discrete channel surface in the sense of \cite{discrete_channel}, if it admits a face-cyclide congruence which is constant along one family of coordinate ribbons. 

In particular, discrete channel surfaces can be characterized by special properties of their curvature sphere congruences \cite[Proposition 2.4]{discrete_channel}: a discrete Legendre map is a discrete channel surface with circular $(1)$-direction if and only if the curvature sphere congruence $s^{(1)}$ is constant along any $(1)$-coordinate line and the curvature spheres $s^{(2)}$ determine a fixed $(2,1)$-plane along each $(1)$-coordinate ribbon.

\begin{theorem}\label{thm_two_channel}
A discrete R-congruence is enveloped by two discrete channel surfaces with circular $(1)$-direction if and only if the map $n^{(2)}$ is constant along each $(1)$-coordinate ribbon and the R-spheres along each $(1)$-coordinate line are curvature spheres of a Dupin cyclide.   
\end{theorem}

\begin{proof}
  Let $r:\mathcal{V}\rightarrow \mathbb{P}(\mathcal{L})$ be a discrete R-congruence enveloped by two discrete channel surfaces $f$ and $\hat{f}$ with circular $(1)$-direction. Then, by Proposition \ref{prop_const_curvsphere_envelope}, the map $\two{n}$ is constant along each $(1)$-coordinate ribbon.

To prove the second property of the discrete R-congruence, let us consider a $(1)$-coordinate ribbon and denote the constant curvature spheres of $f$ and $\hat{f}$ along the two boundary $(1)$-coordinate lines by $s_i, s_{i+1}, \hat{s}_i$ and $\hat{s}_{i+1}$, respectively. Furthermore, by \cite[Proposition 2.4]{discrete_channel}, the other family of curvature spheres of $f$ along this $(1)$-coordinate ribbon lies in a $(2,1)$-plane $D_i$. 

Defining the projection onto $\lspan{\hat{s}_i}^\perp$
\begin{equation*}
\pi(\tau)=\tau - \frac{\lspan{\tau, \hat{s}_i}}{\lspan{s_i, \hat{s}_i}}s_i,
\end{equation*}
we deduce that all R-spheres along the $(1)$-coordinate line lie in the $(2,1)$-plane $\pi(D_i)$. Therefore, the R-spheres are curvature spheres of a Dupin cyclide.

\medskip

Conversely, suppose that the R-spheres along a $(1)$-coordinate line $\gamma_{i_0}$ are curvature spheres of a Dupin cyclide, that is, they lie in a $(2,1)$-plane $C_{i_0}$. Then, any choice of two initial contact elements $f_0:=\spann{r_{i_0}, s_{i_0}}$ and $\hat{f}_0:=\spann{r_{i_0}, \hat{s}_{i_0}}$, where $s_{i_0}, \hat{s}_{i_0} \in C_{i_0}^\perp$ provides two enveloping discrete channel surfaces.
\end{proof}


As an immediate consequence of the 1-parameter freedom in the choice of the initial contact element $f_0$ in the proof of Theorem \ref{thm_two_channel}, we obtain the following corollary: 

\begin{corollary}\label{cor_channel_family}
If a discrete R-congruence admits two discrete channel surfaces as envelopes, then there exists a 1-parameter family of enveloping discrete channel surfaces.
\end{corollary}

\noindent Furthermore, we remark that Theorem \ref{thm_two_channel} reveals how the constructions given in  Subsection \ref{subsection_construction} yield discrete R-congruences admitting a 1-parameter family of discrete channel surfaces in their Ribaucour families. 

In particular, suppose that a discrete R-congruence consists of point spheres and satisfies the conditions of Theorem \ref{thm_two_channel}. Geometrically, those discrete R-congruences are provided by circular nets where the point spheres of one family of coordinate lines lie on circles such that two adjacent ones are related by a M\"obius inversion. Then, by Corollary \ref{cor_channel_family}, there exists a 1-parameter choice of contact elements such that the constructed principal net provides a discrete channel surface, that is, the principal net has indeed a constant curvature sphere along each coordinate line of one family (for details see also \cite{discrete_channel}). 

\medskip

\noindent Using the R-cyclides of a discrete Ribaucour pair given in Definition \ref{def_rib_cyclides}, we observe that the geometric structure of two enveloping discrete channel surfaces is also reflected in the geometry of the vertical faces of the Ribaucour pair: 

\begin{corollary}
A Ribaucour pair consists of two discrete channel surfaces with circular $(1)$-direction if and only if along each vertical $(1)$-coordinate ribbon there exists a constant R-cyclide.
\end{corollary}

\begin{proof}
Suppose that a discrete R-congruence is enveloped by two discrete channel surfaces $f$ and $\hat{f}$ with circular $(1)$-direction. Then the contact elements of each $(1)$-vertical coordinate ribbon provide a coordinate ribbon of a discrete channel surface: according to Theorem \ref{thm_two_channel}, the curvature spheres along the vertical ribbon, namely the R-spheres, lie in a $(2,1)$-plane. Furthermore, the other curvature spheres of the vertical Legendre map, given by the curvature spheres of $f$ and $\hat{f}$, are constant (cf.\,\cite[Proposition~2.4]{discrete_channel}) 

Hence, since the $(1)$-vertical Legendre maps are discrete channel surfaces, there exists a constant face-cyclide along each of these coordinate ribbons, which is then by definition also an R-cyclide of the Ribaucour pair $(f, \hat{f})$.

Conversely, if along each $(1)$-vertical ribbon there exists a constant R-cyclide, the R-spheres along each ribbon lie in a fixed $(2,1)$-plane and, by Proposition \ref{prop_const_curvsphere_envelope}, the map $\one{n}$ is constant along each $(1)$-coordinate ribbon of the Ribaucour pair. Thus, the claim follows from Theorem \ref{thm_two_channel}. 
\end{proof}

To conclude this section we remark on a general property of Ribaucour transforms of discrete channel surfaces. This also gives insights into the geometry of the other envelopes of a Ribaucour family containing a 1-parameter family of discrete channel surfaces.

\begin{prop}
The Ribaucour transforms of a discrete channel surface have a family of discrete spherical curvature lines.
\end{prop}

\begin{proof}
Let $f$ be a discrete channel surface with circular $(1)$-direction and denote by $r$ a discrete R-congruence of $f$. Contemplate a $(1)$-coordinate line with an adjacent coordinate ribbon: we denote by $s^{(1)}_i$ the constant curvature sphere of $f$ and by  $D^{(2)}_{ij}$ the $(2,1)$-plane containing the curvature spheres of the other curvature sphere congruence along the coordinate ribbon. Then, the contact elements of $f$ along this coordinate line lie in the 3-dimensional space $s^{(1)}_i \oplus D^{(2)}_{ij}$. Hence, in particular, the R-spheres of the discrete R-congruence along this coordinate line, as well as the elements $\one{n}$ along the coordinate ribbon, lie in this space. 

Therefore, the contact elements along each $(1)$-coordinate line of any envelope of $r$ lie in a fixed linear sphere complex and, by Proposition  \ref{spherical_osculating}, we indeed obtain an envelope with a family of spherical curvature lines.  
\end{proof}

\medskip

%
%
\bibliographystyle{abbrv}
\bibliography{mybib}

\medskip
\medskip

\begin{minipage}{7cm}
\textbf{Thilo R\"orig}
\\TU Berlin, Institute of Mathematics
\\Secr. MA 8-4, 
\\10623 Berlin, Germany 
\\roerig@math.tu-berlin.de
\end{minipage}
\begin{minipage}{6.6cm}
\textbf{Gudrun Szewieczek}
\\TU Wien
\\Wiedner Hauptstra\ss e 8-10/104 
\\1040 Vienna, Austria
\\gudrun@geometrie.tuwien.ac.at
\end{minipage}

\end{document}

%% file: A-left.tikz
\begin{tikzpicture}[scale=1.4]
\def\radius{1}
\draw (-\radius,-\radius) -- (\radius,-\radius) -- (\radius,\radius) -- (-\radius,\radius) -- (-\radius,-\radius);
\draw [dashed] (-1.5*\radius, 0) -- (1.5*\radius,0);
\draw [dashed] (0, -1.5*\radius) -- (0, 1.5*\radius);
\node [anchor=north] at (0, -1.5) {$\sigma^{(1)}$};
\node [anchor=east] at (-1.5, 0) {$\sigma^{(2)}$};
\node [circle, draw, fill=white] at (-\radius,-\radius) {$r_i$};
\node [circle, draw, fill=white] at (\radius,-\radius) {$r_j$};
\node [circle, draw, fill=white] at (\radius,\radius) {$r_k$};
\node [circle, draw, fill=white] at (-\radius,\radius) {$r_l$};
\end{tikzpicture}

%% file: C-embedded.tikz
\begin{tikzpicture}[scale=1.3]
  \def\radius{1cm}
  \def\ptsize{1pt}
  \coordinate (A) at (45:\radius);
  \coordinate (B) at (175:\radius);
  \coordinate (C) at (315:\radius);
  \coordinate (D) at (5:\radius);
  \draw [fill=white!95!gray] (0,0) circle [radius=\radius];
  \coordinate (n1) at (intersection of A--B and C--D);
  \coordinate (n2) at (intersection of A--D and C--B);
  \coordinate (n3) at (intersection of A--C and D--B);
  \draw [thick] (A)--(B)--(C)--(D)--(A);
  \draw (A)--(n1)--(D);
  \draw (A)--(n2)--(B);
  \draw [fill=white] (n1) circle [radius=\ptsize]; 
  \draw [fill=white] (n2) circle [radius=\ptsize]; 
  \draw [fill=black] (A) circle [radius=\ptsize];
  \draw [fill=black] (B) circle [radius=\ptsize];
  \draw [fill=black] (C) circle [radius=\ptsize];
  \draw [fill=black] (D) circle [radius=\ptsize];
  \draw [fill=white] (n1) circle [radius=\ptsize]; 
  \draw [fill=white] (n2) circle [radius=\ptsize]; 
  \node [anchor=west] at (n1) {$n^{(1)}$};
  \node [anchor=west] at (n2) {$n^{(2)}$};
  \node [anchor=north] at (0,-\radius) {$\operatorname{cr} < 0$};
  \clip (current bounding box.south west) rectangle (current bounding box.north east);
\draw [dashed] (n1) ++(170:1.3) arc (170:260:1.3); 
\draw [dashed] (n2) ++(110:1.65) arc (110:180:1.65); 
\end{tikzpicture}

%% file: C-non_embedded.tikz
\begin{tikzpicture}[scale=1.3]
  \def\radius{1cm}
  \def\ptsize{1pt}
  \coordinate (A) at (90:\radius);
  \coordinate (C) at (220:\radius);
  \coordinate (B) at (360:\radius);
  \coordinate (D) at (50:\radius);
  \draw [fill=white!95!gray] (0,0) circle [radius=\radius];
  \coordinate (n1) at (intersection of A--B and C--D);
  \coordinate (n2) at (intersection of A--D and C--B);
  \draw [thick] (A)--(B)--(C)--(D)--(A);
  \draw (A)--(n1)--(D);
  \draw (A)--(n2)--(B);
  \draw [fill=white] (n1) circle [radius=\ptsize]; 
  \draw [fill=white] (n2) circle [radius=\ptsize]; 
  \draw [fill=black] (A) circle [radius=\ptsize];
  \draw [fill=black] (B) circle [radius=\ptsize];
  \draw [fill=black] (C) circle [radius=\ptsize];
  \draw [fill=black] (D) circle [radius=\ptsize];
  \draw [fill=white] (intersection of A--B and C--D) circle [radius=\ptsize]; 
  \draw [fill=white] (intersection of A--D and C--B) circle [radius=\ptsize]; 
\node [anchor=west] at (n1) {\contour{white}{$n^{(1)}$}};
  \node [anchor=west] at (n2) {$n^{(2)}$};
  \node [anchor=north] at (0,-\radius) {$\operatorname{cr} > 0$};
  \clip (current bounding box.south west) rectangle (current bounding box.north east);
  \draw [dashed] (n2) ++(155:1.65) arc (155:225:1.65); 
\end{tikzpicture}

%% file: B-linecongruence_quad.tikz
\begin{tikzpicture}[rotate=25,scale=\tikzscale]
  \def\radius{1cm}
  \def\ptsize{1pt}
  \coordinate (A) at (-1, 1);
  \coordinate (B) at (-0.2, -0.5);
  \coordinate (C) at (1.1, -0.4);
  \coordinate (D) at (1.5, 0.5);
  \coordinate (n1) at (intersection of A--B and C--D);
  \coordinate (n2) at (intersection of A--D and C--B);
  \coordinate (BC) at ($(B)!0.5!(C)$);
  \coordinate (sBC) at ($(n1)!0.6!(BC)$);
  \coordinate (sAD) at ($(n1)!3.0!(BC)$);
  \coordinate (sAB) at (intersection of sBC--B and sAD--A); 
  \coordinate (sCD) at (intersection of sBC--C and sAD--D); 
  \draw [dashed] (A) -- (n1) -- (D);
  \draw [dashed] (n1) -- (sAD);
  \draw [line width=2*\ptsize, white] (A) -- (B) -- (C) -- (D) -- (A);
  \draw [thick, fill=black!5!white, fill opacity=.5] (A) -- (B) -- (C) -- (D) -- (A);
  \tikzset{line congruence/.style={black!50!green}}
  \draw [line congruence] ($(sAB)!-0.1!(sBC)$) -- ($(sAB)!1.1!(sBC)$);
  \draw [line congruence] ($(sBC)!-0.1!(sCD)$) -- ($(sBC)!1.1!(sCD)$);
  \draw [line congruence] ($(sCD)!-0.1!(sAD)$) -- ($(sCD)!1.1!(sAD)$);
  \draw [line congruence] ($(sAD)!-0.1!(sAB)$) -- ($(sAD)!1.1!(sAB)$);

  \draw [dashed] (B) -- (n2) -- (A);
  \draw [line width=2*\ptsize, white] (sCD) -- (sAB);
  \draw [dashed] (n2) -- (sAB);
  \tikzset{label node/.style={inner sep=0, outer sep=.1cm, circle, fill=white, fill opacity=0.8}}
  \node [anchor=south east] at (A) {$r_i$};
  \node [anchor=north east] at (B) {$r_j$};
  \node [anchor=north west] at (C) {$r_k$};
  \node [anchor=south west] at (D) {$r_l$};

  \node [anchor=west] at (n2) {$n^{(2)}$};
  \node [anchor=west] at (n1) {$n^{(1)}$};
  
  \node [label node, anchor=east] at (sAB) {$s_{ij}$};
  \node [label node, anchor=north] at (sBC) {$s_{jk}$};
  \node [label node, anchor=west] at (sCD) {$s_{kl}$};
  \node [label node, anchor=south] at (sAD) {$s_{il}$};

  \node [label node, anchor=south west] at ($(sCD)!0.5!(sAD)$)
    {\color{black!50!green}$f_l$};
  \node [label node, anchor=south east] at ($(sAD)!0.4!(sAB)$)
    {\color{black!50!green}$f_i$};
  \node [label node, anchor=north east] at ($(sAB)!0.3!(sBC)$)
    {\color{black!50!green}$f_j$};
  \node [label node, anchor=west] at ($(sBC)!0.7!(sCD)$)
    {\color{black!50!green}$f_k$};
  \foreach \sphere in {A, B, C, D, sBC, sAD, sAB, sCD}
  {
    \draw [fill=black] (\sphere) circle [radius=\ptsize];
  }
  \foreach \complex in {n1, n2}
  {
    \draw [fill=white] (\complex) circle [radius=\ptsize];
  }
\end{tikzpicture}

%% file: F-multi_r_congruence.tikz
\begin{tikzpicture}[scale=1.0]
  \def\ptsize{1pt}
  \foreach \m in {0.5,...,2.5}
  {
    \draw [dashed] (\m, -0.5) -- (\m, 3.5);
  }

  \node [anchor=north] at (0.5,-0.5) {$n^{(1)}_{s}$};
  \node [anchor=north] at (1.5,-0.5) {$n^{(1)}_{s+1}$};
  \node [anchor=north] at (2.5,-0.5) {$n^{(1)}_{s+2}$};

  \foreach \n in {0.5,...,2.5}
  {
    \draw [dashed] (-0.5, \n) -- (3.5, \n);
  }

  \node [anchor=east] at (-0.5,0.5) {$n^{(2)}_{t}$};
  \node [anchor=east] at (-0.5,1.5) {$n^{(2)}_{t+1}$};
  \node [anchor=east] at (-0.5,2.5) {$n^{(2)}_{t+2}$};

  \foreach \m in {0,...,3}
  {
    \draw (\m,0)--(\m,3);
  }
  \foreach \n in {0,...,3}
  {
    \draw (0,\n)--(3,\n);
  }
  \foreach \m in {0,...,3}
  {
    \foreach \n in {0,...,3}
    {
      \draw [fill=white] (\m,\n) circle (5pt);
    }
  }
\end{tikzpicture}
%
%
%
%
%
%
%
%
%
%

%% file: H-sigma3.tikz
\begin{tikzpicture}[scale=3.5]
  \tikzset{label node/.style={minimum size=.6cm,inner sep=.5pt, outer sep=.1cm, circle, fill=white, fill opacity=0.8}}

  \coordinate (v1) at (1.0,0);
  \coordinate (v2) at (30:0.8);
  \coordinate (v3) at (0,1);
  \coordinate (v1h) at ($0.5*(v1)$);
  \coordinate (v2h) at ($0.5*(v2)$);
  \coordinate (v3h) at ($0.5*(v3)$);



  \draw (v2) -- ++ (v3);
  \draw[white, thick] (0,0) -- ++ (v1) -- ++ (v2) -- ++ ($-1*(v1)$) -- ++ ($-1*(v2)$);
  \draw[black!50!green] (0,0) -- ++ (v1) -- ++ (v2) -- ++ ($-1*(v1)$) -- ++ ($-1*(v2)$);

  \draw[white, thick] (0,0) -- ++ (v3) (v1) -- ++ (v3);
  \draw (0,0) -- ++ (v3) (v1) -- ++ (v3);

  \draw[white, thick] ($(v1) + (v2)$) -- ++ (v3);
  \draw ($(v1) + (v2)$) -- ++ (v3);

  \draw[white, thick] (v3) -- ++ (v1) -- ++ (v2) -- ++ ($-1*(v1)$) -- ++ ($-1*(v2)$);
  \draw[black!50!green] (v3) -- ++ (v1) -- ++ (v2) -- ++ ($-1*(v1)$) -- ++ ($-1*(v2)$);

  \draw[dotted] ($(v2h) - 0.25*(v1)$) -- ++ ($1.5*(v1)$); 
  \draw[dotted] ($(v2h) - 0.25*(v1) + 1.5*(v1)$) -- ++ (v3); 
  \draw[dotted] ($(v3) + (v2h) - 0.25*(v1)$) -- ++ ($1.5*(v1)$); 
  \draw[dotted] ($(v2h) - 0.25*(v1)$) -- ++ (v3); 
  \node[anchor=south east] at ($0.5*(v2)-0.25*(v1) + (v3)$) {$\two{n}$};

  \draw[dotted] ($(v1h) - 0.25*(v2)$) -- ++ ($1.5*(v2)$); 
  \draw[dotted] ($(v1h) - 0.25*(v2) + 1.5*(v2)$) -- ++ (v3); 
  \draw[dotted] ($(v3) + (v1h) - 0.25*(v2)$) -- ++ ($1.5*(v2)$); 
  \draw[dotted] ($(v1h) - 0.25*(v2)$) -- ++ (v3); 
  \node[anchor=south west] at ($0.5*(v1) + 1.25*(v2) + (v3)$) {$\one{n}$};

  \draw[dashed] ($(v3h) - 0.25*(v1)$) -- ++ ($1.5*(v1)$); 
  \draw[dashed] ($(v3h) - 0.25*(v1) + 1.5*(v1)$) -- ++ (v2); 
  \draw[dashed] ($(v2) + (v3h) - 0.25*(v1)$) -- ++ ($1.5*(v1)$); 
  \draw[dashed] ($(v3h) - 0.25*(v1)$) -- ++ (v2); 
  \node[anchor=west] at ($1.25*(v1) + 0.5*(v3) + (v2)$) {$\three{n}$};

  \node [anchor=south east] at ($(v2) + (v3)$) {\color{black!50!green}$f$};
  \node [anchor=north east] at (0,0) {\color{black!50!green}$\hat{f}$};

  \path (0,0) -- ++ (v3) node[label node, circle, fill=black!10!white, draw, midway] {$r_i$}; 
  \path (v1) -- ++ (v3) node[label node, circle, fill=black!10!white, draw, midway] {$r_j$}; 
  \path ($(v2) + (v1)$) -- ++ (v3) node[label node, circle, fill=black!10!white, draw, midway] {$r_k$}; 
  \path (v2) -- ++ (v3) node[label node, circle, fill=black!10!white, draw, midway] {$r_l$}; 

  \path (0,0) -- ++ (v1) 
  node[label node, circle, fill=white, draw=black!50!green, midway] {$\hat{s}_{ij}$}; 
  \path (v1) -- ++ (v2) 
  node[label node, circle, fill=white, draw=black!50!green, midway] {$\hat{s}_{jk}$}; 
  \path (v2) -- ++ (v1) 
  node[label node, circle, fill=white, draw=black!50!green, midway] {$\hat{s}_{kl}$}; 
  \path (0,0) -- ++ (v2) 
  node[label node, circle, fill=white, draw=black!50!green, midway] {$\hat{s}_{il}$}; 
  
  \path ($(0,0) + (v3)$) -- ++ (v1) 
  node[label node, circle, fill=white, draw=black!50!green, midway] {${s}_{ij}$}; 
  \path ($(v1) + (v3)$) -- ++ (v2) 
  node[label node, circle, fill=white, draw=black!50!green, midway] {${s}_{jk}$}; 
  \path ($(v2) + (v3)$) -- ++ (v1) 
  node[label node, circle, fill=white, draw=black!50!green, midway] {${s}_{kl}$}; 
  \path ($(0,0) + (v3)$) -- ++ (v2) 
  node[label node, circle, fill=white, draw=black!50!green, midway] {${s}_{il}$}; 
\end{tikzpicture}

%% file: notation.tikz
\begin{tikzpicture}[scale=1.0]
  \def\qsize{1.5}
  \tikzset{
    label node/.style= {
      inner sep=0, outer sep=.1cm, circle, fill=white, fill opacity=0.8
    }
  }
  \draw [xstep=\qsize, ystep=\qsize] (-\qsize,-\qsize) grid (\qsize,\qsize);
  \node [label node, anchor=north east] at (-\qsize, -\qsize) {$i$};
  \node [label node, anchor=east] at (-\qsize, 0) {$i'$};
  \node [label node, anchor=south east] at (-\qsize, \qsize) {$i''$};

  \node [label node, anchor=north] at (0, -\qsize) {$j$};
  \node [label node, anchor=north west] at (0, 0) {$j'$};
  \node [label node, anchor=south] at (0, \qsize) {$j''$};

  \node [label node, anchor=north west] at (\qsize, -\qsize) {$k$};
  \node [label node, anchor=west] at (\qsize, 0) {$k'$};
  \node [label node, anchor=south west] at (\qsize, \qsize) {$k''$};

  \node [label node, anchor=south west] at ($0.5*(-\qsize, -\qsize)$) {$d_\alpha$};
  \node [label node] at ($0.5*(\qsize, -\qsize)$) {$d_\beta$};
  \node [label node] at ($0.5*(\qsize, \qsize)$) {$d_\gamma$};
  \node [label node] at ($0.5*(-\qsize, \qsize)$) {$d_\delta$};

  \draw [dashed] ($0.5*(-\qsize,0)+0.66*(-\qsize,-\qsize)$) --
  ++ ($1.33*(\qsize,0)$);
  \node [label node, anchor=east] at ($0.5*(-\qsize,0)+0.66*(-\qsize,-\qsize)$) {$n_\alpha^{(2)}$};
  \draw [dashed] ($0.5*(0,-\qsize)+0.66*(-\qsize,-\qsize)$) --++ ($1.33*(0,\qsize)$);
  \node [label node, anchor=north] at ($0.5*(0,-\qsize)+0.66*(-\qsize,-\qsize)$) {$n_\alpha^{(1)}$};
\end{tikzpicture}

%% file: main.bbl
\begin{thebibliography}{10}

\bibitem{blaschke}
W.~{Blaschke}.
\newblock {\em {Vorlesungen {\"u}ber Differentialgeometrie {III}}}.
\newblock Springer Grundlehren XXIX, Berlin, 1929.

\bibitem{cas}
P.~Bo, H.~Pottmann, M.~Kilian, W.~Wang, and J.~Wallner.
\newblock Circular arc structures.
\newblock {\em ACM Trans. Graphics}, 30:\#101,1--11, 2011.
\newblock Proc. SIGGRAPH.

\bibitem{paper_cyclidic}
A.~{Bobenko} and E.~{Huhnen-Venedey}.
\newblock {Curvature line parametrized surfaces and orthogonal coordinate
  systems: discretization with Dupin cyclides}.
\newblock {\em {Geom. Dedicata}}, 159:207--237, 2012.

\bibitem{multinet}
A.~{Bobenko}, H.~{Pottmann}, and T.~{R{\"o}rig}.
\newblock {Multi-Nets. Classification of discrete and smooth surfaces with
  characteristic properties on arbitrary parameter rectangles}.
\newblock {\em Discrete Comput Geom}, 63:624--655, 2020.

\bibitem{ddg_book}
A.~{Bobenko} and Y.~{Suris}.
\newblock {Discrete differential geometry: Integrable structure}.
\newblock {\em {Graduate Studies in Mathematics 98, Amer. Math. Soc.,
  Providence}}, 2008.

\bibitem{bobenko_schief-discrete_line_complexes}
A.~I. Bobenko and W.~K. Schief.
\newblock Discrete line complexes and integrable evolution of minors.
\newblock {\em Proceedings of the Royal Society A: Mathematical, Physical and
  Engineering Sciences}, 471(2175):20140819, 2015.

\bibitem{org_principles}
A.~I. {Bobenko} and Y.~B. {Suris}.
\newblock {On organizing principles of discrete differential geometry. Geometry
  of spheres.}
\newblock {\em {Russ. Math. Surv.}}, 62(1):1--43, 2007.

\bibitem{discrete_omega}
F.~{Burstall}, J.~{Cho}, U.~{Hertrich-Jeromin}, M.~{Pember}, and W.~{Rossman}.
\newblock {Discrete $\Omega$-surfaces}.
\newblock in preparation.

\bibitem{MR2254053}
F.~Burstall and U.~Hertrich-Jeromin.
\newblock The {R}ibaucour transformation in {L}ie sphere geometry.
\newblock {\em Differential Geom. Appl.}, 24(5):503--520, 2006.

\bibitem{rib_coord}
F.~{Burstall}, U.~{Hertrich-Jeromin}, and M.~{Lara Miro}.
\newblock {Ribaucour coordinates.}
\newblock {\em {Beitr. Algebra Geom.}}, 60(1):39--55, 2019.

\bibitem{lin_weingarten_discrete}
F.~{Burstall}, U.~{Hertrich-Jeromin}, and W.~{Rossman}.
\newblock {Discrete linear Weingarten surfaces}.
\newblock {\em Nagoya Mathematical Journal}, 231:55–88, 2018.

\bibitem{discrete_cmc}
F.~{Burstall}, U.~{Hertrich-Jeromin}, W.~{Rossman}, and S.~{Santos}.
\newblock Discrete surfaces of constant mean curvature.
\newblock {\em {RIMS Kokyuroku}}, 1880:113--179, 2014.

\bibitem{zbMATH06406008}
F.~{Burstall}, U.~{Hertrich-Jeromin}, W.~{Rossman}, and S.~{Santos}.
\newblock {Discrete special isothermic surfaces.}
\newblock {\em {Geom. Dedicata}}, 174:1--11, 2015.

\bibitem{book_cecil}
T.~{Cecil}.
\newblock {\em {Lie sphere geometry. With applications to submanifolds.}}
\newblock Springer, New York, 2008.

\bibitem{spherical_curv_lines}
J.~{Cho}, M.~{Pember}, and G.~{Szewieczek}.
\newblock {Surfaces with one family of spherical curvature lines.}
\newblock in preparation.

\bibitem{MR2103311}
A.~V. Corro and K.~Tenenblat.
\newblock Ribaucour transformations revisited.
\newblock {\em Comm. Anal. Geom.}, 12(5):1055--1082, 2004.

\bibitem{MR2320656}
M.~Dajczer, L.~A. Florit, and R.~Tojeiro.
\newblock The vectorial {R}ibaucour transformation for submanifolds and
  applications.
\newblock {\em Trans. Amer. Math. Soc.}, 359(10):4977--4997, 2007.

\bibitem{MR1901374}
M.~Dajczer and R.~Tojeiro.
\newblock An extension of the classical {R}ibaucour transformation.
\newblock {\em Proc. London Math. Soc. (3)}, 85(1):211--232, 2002.

\bibitem{MR2028680}
M.~Dajczer and R.~Tojeiro.
\newblock Commuting {C}odazzi tensors and the {R}ibaucour transformation for
  submanifolds.
\newblock {\em Results Math.}, 44(3-4):258--278, 2003.

\bibitem{DOLIWA1999169}
A.~Doliwa.
\newblock Quadratic reductions of quadrilateral lattices.
\newblock {\em Journal of Geometry and Physics}, 30(2):169--186, 1999.

\bibitem{zbMATH01721625}
A.~{Doliwa}.
\newblock {The Ribaucour congruences of spheres within Lie sphere geometry.}
\newblock In {\em {B\"acklund and Darboux transformations. The geometry of
  solitons.}}, pages 159--166. AMS, 2001.

\bibitem{DoliwaSantiniManas:2000:TransformationsOfQnets}
A.~Doliwa, P.~M. Santini, and M.~Ma{\~n}as.
\newblock {Transformations of quadrilateral lattices}.
\newblock {\em J. Math. Phys.}, 41(2):944--990, 2000.

\bibitem{MR2004958}
U.~Hertrich-Jeromin.
\newblock {\em Introduction to {M}\"{o}bius differential geometry}.
\newblock London Mathematical Society Lecture Note Series. Cambridge University
  Press, Cambridge, 2003.

\bibitem{MR1676683}
U.~Hertrich-Jeromin, T.~Hoffmann, and U.~Pinkall.
\newblock A discrete version of the {D}arboux transform for isothermic
  surfaces.
\newblock In {\em Discrete integrable geometry and physics}, volume~16 of {\em
  Oxford Lecture Ser. Math. Appl.}, pages 59--81. Oxford Univ. Press, New York,
  1999.

\bibitem{discrete_channel}
U.~Hertrich-Jeromin, W.~Rossman, and G.~Szewieczek.
\newblock Discrete channel surfaces.
\newblock {\em Mathematische Zeitschrift}, 294:747--767, 2020.

\bibitem{zbMATH01327008}
B.~G. {Konopelchenko} and W.~K. {Schief}.
\newblock {Three-dimensional integrable lattices in Euclidean spaces: conjugacy
  and orthogonality.}
\newblock {\em {Proc. R. Soc. Lond., Ser. A, Math. Phys. Eng. Sci.}},
  454(1980):3075--3104, 1998.

\bibitem{O_surface}
B.~G. {Konopelchenko} and W.~K. {Schief}.
\newblock {On the unification of classical and novel integrable surfaces. {I}.
  {D}ifferential geometry}.
\newblock {\em R.~Soc.~Lond.~Proc.~Math.~Phys.~Eng.~Sci.}, 459(2029):67--84,
  2003.

\bibitem{trafo_channel}
M.~{Pember} and G.~{Szewieczek}.
\newblock {Channel surfaces in Lie sphere geometry}.
\newblock {\em Beitr Algebra Geom}, 59:779--796, 2018.

\bibitem{saji2020behavior}
K.~Saji and K.~Teramoto.
\newblock Behavior of principal curvatures of frontals near non-front singular
  points and their application, 2020.
\newblock arXiv:2003.07256v1.

\bibitem{MR1997461}
W.~K. Schief.
\newblock On the unification of classical and novel integrable surfaces. {II}.
  {D}ifference geometry.
\newblock {\em R. Soc. Lond. Proc. Ser. A Math. Phys. Eng. Sci.},
  459(2030):373--391, 2003.

\bibitem{MR2251001}
K.~Tenenblat and Q.~Wang.
\newblock Ribaucour transformations for hypersurfaces in space forms.
\newblock {\em Ann. Global Anal. Geom.}, 29(2):157--185, 2006.

\bibitem{MR2743444}
C.-L. Terng.
\newblock Geometric transformations and soliton equations.
\newblock In {\em Handbook of geometric analysis, {N}o. 2}, volume~13 of {\em
  Adv. Lect. Math.}, pages 301--358. Int. Press, Somerville, MA, 2010.

\end{thebibliography}
